\documentclass{amsart}
\usepackage{graphicx}
\usepackage{graphicx,esint}
\usepackage{amssymb,amscd,amsthm,amsxtra}
\usepackage{latexsym}
\usepackage{epsfig}
\usepackage{color} 
\newtheorem{thm}{Theorem}[section]

\newtheorem{lem}[thm]{Lemma}

\theoremstyle{definition}
\newtheorem{defn}[thm]{Definition}
\theoremstyle{remark}
\newtheorem{rem}[thm]{\bf Remark}
\numberwithin{equation}{section}
\newcommand{\R}{\mathbb R}
\newcommand{\be}{\begin{equation}}
\newcommand{\ee}{\end{equation}}

\newcommand{\ep}{\eps}

\newcommand{\eps}{\varepsilon}

\newcommand{\p}{\partial}

\newcommand{\comment}[1]{}
\begin{document}

\title[Lipschitz regularity to two-phase $p$-Laplacian problems]{Lipschitz regularity of solutions to two-phase $p$-Laplacian free boundary problems with right hand side}

\author[Fausto Ferrari]{Fausto Ferrari}
\address{Dipartimento di Matematica dell'Universit\`a di Bologna, Piazza di Porta S. Donato, 5, 40126 Bologna, Italy.}
\email{\tt fausto.ferrari@unibo.it}
\author[Claudia Lederman]{Claudia Lederman}
\address{IMAS - CONICET and Departamento  de
Ma\-te\-m\'a\-ti\-ca, Facultad de Ciencias Exactas y Naturales,
Universidad de Buenos Aires, (1428) Buenos Aires, Argentina.}
\email{\tt  clederma@dm.uba.ar}
\thanks{This work was partially
supported by a grant from the IMU-CDC and Simons Foundation}
\thanks{This research was partially funded by
 INDAM-GNAMPA grant CUP E53C23001670001 for visiting professors}
\thanks{F. F. was partially supported by
 PRIN 2022 7HX33Z - CUP J53D23003610006, Pattern formation in nonlinear phenomena,  INDAM-GNAMPA 2022 project: {\it Regolarit\`a locale e globale per
problemi completamente
non lineari} and INDAM-GNAMPA 2024 project: {\it Free boundary
problems in noncommutative structures and degenerate operators} CUP E53C23001670001.  }
\thanks{C. L. was partially supported   by the grants CONICET PIP 11220220100476CO 2023-2025,  UBACYT 20020220200056BA and ANPCyT PICT 2019-00985. C. L. wishes to thank the Department of Mathematics of the University of Bologna, Italy, for the kind hospitality.}


\begin{abstract}
We prove the local Lipschitz continuity of viscosity solutions for two-phase free boundary problems for the $p$-Laplacian with non-zero right hand side, where $p\in (1,\infty)$. This is the optimal regularity for the problem.

We also obtain the local H\"older continuity for a larger class of problems.

The results introduced here are new even in the homogeneous situation, that is, when the right hand side is zero.

Our work applies to merely viscosity solutions, which allows a wide applicability.

\end{abstract}

% ----------------------------------------------------------------
\maketitle
% ----------------------------------------------------------------

\section{Introduction and main results}\label{section1}

In this article we obtain the local Lipschitz continuity of viscosity solutions to  two-phase free boundary problems governed by the 
$p$-Laplacian  with non-zero right hand side.
This continues our work in \cite{FL1}, \cite{FL2} and \cite{FL3} (see \cite{FLS} as well).

More precisely, we denote by $$\Delta_{p} u:=\mbox{div} (|\nabla u|^{p-2}\nabla u),$$
where $p$ is a constant, $1<p<\infty$. Then, the problem we investigate here is  the following: 
\begin{equation}  \label{fbtrue}
\left\{
\begin{array}{ll}
\Delta_{p} u = f, & \hbox{in $\Omega^+(u)\cup \Omega^-(u)$}, \\
\  &  \\
u_\nu^+=G(u_\nu^-,x), & \hbox{on $F(u):= \partial \Omega^+(u) \cap
\Omega,$} 
\end{array}
\right.
\end{equation}
where $\Omega  \subset \mathbb{R}^n$ is a bounded domain and
$$\Omega^+(u):= \{x \in \Omega : u(x)>0\}, \qquad\quad \Omega^-(u):=\{x \in \Omega : u(x)\leq 0\}^{\circ },$$
 while $u_\nu^+$ and $u_\nu^-$ denote the normal derivatives in the inward
direction to $\Omega^+(u)$ and $\Omega^-(u)$ respectively. $F(u)$ is called the \emph{free boundary} of $u$.
Furthermore, $f\in L^{\infty }(\Omega )$ is continuous in $\Omega^+(u) \cup \Omega^-(u)$.

The function 
$$G(t,x):[0,\infty)\times\Omega\rightarrow(0,\infty),$$
$G(\cdot, x)\in C^2(0,\infty)$ for $x\in \Omega$, 
satisfies:

\begin{itemize}
\item[(P1)]
$G(\cdot,x)$ is strictly increasing for $x\in \Omega$  and $G(t,x)\to \infty$ as $t\to \infty$, uniformly in $x\in \Omega$.
\end{itemize}

In our main result we assume that $G(t,x)$ behaves like $t$, for $t$ large, uniformly in $x$. More precisely, we assume that :

\begin{itemize}
\item[(P2)] $\frac{\partial G(t,x)}{\partial t}\to 1$, $\frac{\partial^2G(t,x)}{\partial t^2}=O(\frac{1}{t})$, as $t\to\infty$, uniformly in $x\in \Omega$,

\item[(P3)] $G(t,\cdot)\in C^{0,\bar\gamma}(\Omega)$ uniformly in $t$, for some $0<\bar{\gamma}<1$.
\end{itemize}

When $f\not\equiv 0$ and $p\neq 2$ we assume as well that:
\begin{itemize}
\item[(P4)] $G(t,x)\equiv G(t)$, with $t^{-1}G(t)$ strictly decreasing, for $t$ large.
\end{itemize}

This research includes $G(t, x)=(g(x)+t^p)^{\frac{1}{p}}$, $p>1$, where $g\in C^{0,\bar{\gamma}}(\Omega)$, for some $0<\bar{\gamma}<1$, $g>0$,  which arises in several applications (see for instance, Appendix A in \cite{FL3}).

Our assumptions on $f$ are
\begin{equation}\label{assump-f}
f\in L^{\infty}(\Omega),\qquad f\text{ is continuous in }\Omega^+(u)\cup \Omega^-(u).
\end{equation}
Nevertheless, the arguments also apply when  $f$ is merely bounded measurable, but we require \eqref{assump-f} to avoid technicalities.

The main result in this work is the following one (for notation and the precise definition of viscosity solution to \eqref{fbtrue} we refer to Section \ref{section2}).
\begin{thm}[Optimal regularity]\label{Lipschitz_cont}
Let $u$ be a viscosity solution to \eqref{fbtrue} in $B_1$. Assume that \eqref{assump-f} holds in $B_1$ and $G$ satisfies assumptions (P1), 
(P2) and (P3) in $B_1$. If $f\not\equiv 0$ and $p\neq 2$   assume that also (P4) holds, then 
\begin{equation*}
\|\nabla u\|_{L^\infty(B_{1/2})}\leq C (\|u\|_{L^\infty(B_{3/4})}+1),
\end{equation*}
where $C=C(n,p,\|f\|_{L^\infty(B_1)}, G)$ is a positive constant.
\end{thm}

Let us point out that this is the optimal regularity for the problem.

We here follow the approach introduced in  \cite{DS} for uniformly elliptic fully nonlinear operators.

The heuristics behind the proof of Theorem \ref{Lipschitz_cont} is that, in the regime of {\it big gradients}, the free boundary condition becomes a continuity  condition for the gradient (i.e., no-jump). 
As a consequence, the desired gradient bound follows from interior $C^{1,\alpha}$ estimates for the $p$-Laplace operator and from free boundary regularity results for two-phase viscosity solutions to \eqref{fbtrue} proven in our recent paper \cite{FL3}.

In the research on two-phase free boundary problems a central question is the optimal regularity of  solutions. In \cite{CJK} the authors proved  a monotonicity formula that allows to obtain the local Lipschitz continuity of solutions to \eqref{fbtrue} in the case of $p=2$, that is, when the equation is linear. {\it No such monotonicity formula is known up to now for the general case $p\neq 2$}.

We remark that in the present article  we are concerned with the question of local Lipschitz continuity of {\it viscosity solutions} for two-phase free boundary problems for the $p$-Laplacian {\it with non-zero right hand side}.

The results introduced here are new even in the homogeneous situation, that is, when $f\equiv 0$, since, to our knowledge, the only regularity result existing in literature, about Lipschitz continuity across the free boundary, concerns minima of the $p$-Bernoulli functional in the homogeneous case, see \cite{DK}.

Our work applies to merely viscosity solutions, which allows a wide applicability.

We point out that in the one  phase case, i.e., when $u\ge 0$, the local Lipschitz continuity of solutions to  \eqref{fbtrue} follows from our previous paper \cite{FL2}, where a more general one phase free boundary problem was treated.

We remark that carrying out for the inhomogeneous $p$-Laplace operator a strategy originally devised  for
 {\it uniformly elliptic operators} presents
challenging difficulties due to the type of nonlinear behavior of the $p$-Laplacian. This was already the case in our previous works \cite{FL1}, \cite{FL2} and \cite{FL3} for the treatment of this problem.  In fact, the $p$-Laplacian  is a nonlinear operator that appears naturally in divergence structure from minimization problems, i.e., in the form ${\rm div}A(\nabla u)=f(x)$, with
\begin{equation}\label{Fan-1.6}
\lambda |\eta|^{p-2}|\xi|^2\le\sum_{i,j=1}^n\frac{\partial A_i}{\partial \eta_j}(\eta)\xi_i\xi_j\le \Lambda |\eta|^{p-2}|\xi|^2, \quad \xi\in\R^n,
\end{equation}
where $0<\lambda\le\Lambda$. This operator is singular  when $1<p<2$  and  degenerate when $p>2$. Dealing with this problem   in the presence of {\it a non-zero right hand side} is very delicate since,  in this case, {\it we can not neglect the factor  $|\eta|^{p-2}$ in \eqref{Fan-1.6}}. 

In particular, the discussion at the end of Section 3.2 in \cite{DS}, concerning the generalization of their results to some equations with zero right hand side, {\it does not apply to our inhomogeneous problem  \eqref{fbtrue}}.

The process of obtaining the local Lipschitz continuity of viscosity solutions to \eqref{fbtrue} we follow requires to prove first the local H\"older continuity of viscosity solutions to this problem ---which is interesting by itself. In fact, we get this result under more general conditions on $G$. The precise result  is the following: 
\begin{thm}[Local H\"older continuity]\label{holder_reg}
Let $u$ be a viscosity solution to \eqref{fbtrue} in $B_1$. Assume that, for $\sigma>0$ and $M>0$, there holds
\begin{equation}\label{G2}
\sigma t\leq G(t, x)\leq \sigma^{-1}t,\quad \mbox{for}\:\:t>M,\: x\in B_1.
\end{equation}
Then $u\in C^{0,\alpha}(B_{1/2})$, for some $0<\alpha<1$, and
$$
\|u\|_{C^{0,\alpha}(B_{1/2})}\leq C(\|u\|_{L^{\infty}(B_{3/4})}+1),
$$
where  $\alpha$ depends only $\sigma, n$ and $p$,  and $C>0$ depends only $\sigma, M, n, p$ and $\|f\|_{L^{\infty}(B_1)}$. 
\end{thm}

An important contribution of this article can be found in Section \ref{section3}, where we revisit \cite{FL3}. More precisely, we refine a $C^{1,\gamma}$ regularity result for flat free boundaries we proved in  \cite{FL3}, obtaining Theorem \ref{main_new-nondeg-general}. We develop, in addition, some improvements of it for the particular cases $f\equiv 0$, $p=2$ and $G(t,x)\equiv G(t)$. See Theorems \ref{main_new-nondeg-general-rhs0}, \ref{main_new-nondeg-general-p=2} and \ref{main_new-nondeg-general-G(t)} ---which are  key tools in the proof of Theorem \ref{Lipschitz_cont}. In this way, we are able to get Theorem  \ref{Lipschitz_cont} for very general free boundary conditions.
Moreover, we obtain some preliminary results, based on these theorems, that play a fundamental role in the proof of Theorem \ref{Lipschitz_cont} (see Remark \ref{further_remark}).

Let us mention that when the problem is truly nonlinear singular / degenerate with a non-zero right hand side ---i.e., when $f\not\equiv 0$ and $p\neq 2$--- due to the difficulties inherent to this situation, we are  assuming in our main result that (P4) also holds. We point out that assumptions of these type appear frequently in the study of free boundary problems. See, for instance, \cite{C1, C2, DFS1, DFS2}.

Let us also stress that we include here the case in which $p=2$ ---i.e., the case of Laplace operator--- relying on our work \cite{FL3} (see Section \ref{section3}). Notice that, although \cite{DS} includes Laplace operator, the case of a non-zero right hand side and a general free boundary condition $u_\nu^+=G(u_\nu^-,x)$ was not explicitly developed there. Hence, Theorem \ref{Lipschitz_cont}   gives an explicit  proof for the case $p=2$ which is different from the one that can be deduced from the monotonicity formula of \cite{CJK} ---{\it only available for this particular situation}.

As already mentioned, the main tool in the approach we follow in order to prove the local Lipschitz continuity, inspired by \cite{DS}, is a regularity result for flat free boundaries  for viscosity solution to free boundary problem \eqref{fbtrue}. Let us point out that in the uniformly elliptic linear and fully nonlinear cases, the corresponding results were obtained in \cite{DFS1} and \cite{DFS2} respectively. It is worth mentioning that in these papers it was assumed, all along the works,  the  Lipschitz continuity of the viscosity solutions under consideration. However, in \cite{DS} it was stated a theorem ---that can be derived from the results in \cite{DFS2}--- concerning the free boundary regularity for the fully nonlinear uniformly elliptic case, which holds under a nondegenerate situation, that does not assume the Lipschitz continuity of solutions.

The present article makes use of a free boundary regularity theorem for viscosity solutions to problem \eqref{fbtrue}, that we proved in \cite{FL3}, where the Lipschitz continuity of solutions was not assumed.  We believe that this will contribute to the comprehension of the powerful tools developed in \cite{DS} ---both for the problem studied in the present paper and for the uniformly elliptic case studied in \cite{DS}.

Let us emphasize once more that the fact that we are dealing with a nonlinear singular / degenerate operator with a non-zero right hand side introduces challenging difficulties that we overcome in the present research, with the aid of the companion paper \cite{FL3}.

The paper is organized as follows.
In Section \ref{section2}, we provide basic definitions and notation. In Section \ref{section3}, we revisit a $C^{1,\gamma}$ regularity result for flat free boundaries of viscosity solutions to \eqref{fbtrue}  we proved in \cite{FL3} (see Theorem \ref{main_new-nondeg-general}).  We also develop some improvements of it for the cases $f\equiv 0$, $p=2$ and $G(t,x)\equiv G(t)$ 
(i.e., Theorems \ref{main_new-nondeg-general-rhs0}, \ref{main_new-nondeg-general-p=2} and \ref{main_new-nondeg-general-G(t)}). Moreover, we obtain some preliminary results (see Remark \ref{further_remark}).
 In Section \ref{section4}
we prove Theorem \ref{holder_reg}, which gives the local H\"older continuity of viscosity solutions  for a general class of free boundary problems of type \eqref{fbtrue}.  Finally, in Section \ref{section5}, we prove the local Lipschitz continuity of viscosity solutions to \eqref{fbtrue} namely, Theorem \ref{Lipschitz_cont}. 

\begin{subsection}{Assumptions}\label{assump}

\smallskip

Throughout the paper we let $\Omega\subset\R^n$  be a bounded domain.

\bigskip

\noindent{\bf Assumptions on $p$.} We
assume that
\begin{equation*}
1<p<\infty.
\end{equation*}

\bigskip

\noindent{\bf Assumptions on $f$.} We  assume that 
\begin{equation}\label{assump-f-2}
f\in L^{\infty}(\Omega),\qquad f\text{ is continuous in }\Omega^+(u)\cup \Omega^-(u).
\end{equation}
Nevertheless, the arguments also hold when $f$ is merely bounded measurable, but we require \eqref{assump-f-2} to avoid technicalities.
                                                                                           
\bigskip

\noindent{\bf Assumptions on $G$.}  We assume that  $$G(t,x):[0,\infty)\times\Omega\rightarrow(0,\infty),$$
$G(\cdot, x)\in C^2(0,\infty)$ for $x\in \Omega$,
satisfies 
\begin{itemize}
\item[(P1)]

$G(\cdot,x)$ is strictly increasing for $x\in \Omega$  and $G(t,x)\to \infty$ as $t\to \infty$, uniformly in $x\in \Omega$.
\end{itemize}

In our main result we assume that:

\begin{itemize}
\item[(P2)] $\frac{\partial G(t,x)}{\partial t}\to 1$, $\frac{\partial^2G(t,x)}{\partial t^2}=O(\frac{1}{t})$, as $t\to\infty$, uniformly in $x\in \Omega$,

\item[(P3)] $G(t,\cdot)\in C^{0,\bar\gamma}(\Omega)$ uniformly in $t$, for some $0<\bar{\gamma}<1$.
\end{itemize}

When $f\not\equiv 0$ and $p\neq 2$ we assume as well that:
\begin{itemize}
\item[(P4)] $G(t,x)\equiv G(t)$, with $t^{-1}G(t)$ strictly decreasing, for $t$ large.
\end{itemize}

\end{subsection}

\section{Basic definitions, notation and preliminaries}\label{section2}

In this section, we provide notation, basic definitions and some preliminaries that will be relevant for  our work.

\medskip

\noindent {\bf Notation.} For any continuous function $u:\Omega\subset \mathbb{R}^n\to \mathbb{R}$ we denote
\begin{equation*}
\Omega^+(u):= \{x \in \Omega : u(x)>0\},\quad \Omega^-(u):=\{x \in \Omega : u(x)\leq 0\}^{\circ }
\end{equation*} 
and
\begin{equation*}
F(u):= \partial \Omega^+(u) \cap \Omega. 
\end{equation*} 
We refer to the set $F(u)$ as the {\it free boundary} of $u$, while $\Omega^+(u)$ is its {\it positive phase} (or {\it side}) and $\Omega^-(u)$ is the {\it nonpositive phase}.

We begin with some remarks on the $p$-Laplacian. In particular, we  recall the relationship between the different notions of solutions to $\Delta_{p}u=f$ we are using, namely, weak and viscosity solutions. Then we give the definition of viscosity solution to problem \eqref{fbtrue}  and we deduce some consequences. We here refer to the usual definition of $C$-viscosity  sub/supersolution and solution of an elliptic PDE, see e.g., \cite{CIL}.

We start by observing that 
direct calculations show that, for $C^2$ functions $u$ such that $\nabla u(x)\not=0$ in some open set,
\begin{equation}\label{equation_nondivergence_form}
\begin{split}
&\Delta_{p}u=\mbox{div} (|\nabla u|^{p-2}\nabla u)=|\nabla u|^{p-2}\left(\Delta u+(p-2)\Delta_\infty^Nu\right),
\end{split}
\end{equation}
where
$$
\Delta_\infty^Nu:=\Big\langle D^2 u\frac{\nabla u}{|\nabla u|}\,,\,\frac{\nabla u}{|\nabla u|}\Big\rangle
$$
denotes the normalized $\infty$-Laplace operator.

We deduce, in addition, that
\begin{equation}\label{p(x)-vs-pucci}
\begin{split}
&|\nabla u|^{p-2}\mathcal{M}_{\lambda_0,\Lambda_0}^-(D^2u)
\leq \Delta_{p}u\leq |\nabla u|^{p-2}\mathcal{M}_{\lambda_0,\Lambda_0}^+(D^2u),
\end{split}
\end{equation}
where $\lambda_0:=\min\{1,p-1\}$ and $\Lambda_0:=\max\{1,p-1\}.$  
As usual, if $0<\lambda\leq \Lambda$ are numbers, and $e_i$ is the $i-$th  eigenvalue of the $n\times n$ symmetric matrix $M,$ then $\mathcal{M}_{\lambda,\Lambda}^+$ and $\mathcal{M}_{\lambda,\Lambda}^-$ denote the extremal Pucci operators and are defined  (see \cite{CC}) as
\begin{equation*}
\begin{split}
\mathcal{M}_{\lambda,\Lambda}^+(M)&=\lambda\sum_{e_i<0}e_i+\Lambda\sum_{e_i>0}e_i,\\
\mathcal{M}_{\lambda,\Lambda}^-(M)&=\Lambda\sum_{e_i<0}e_i+\lambda\sum_{e_i>0}e_i.
\end{split}
\end{equation*}

\medskip

First we need 
\begin{defn}\label{defnweak} 
Assume that $f\in L^{\infty}(\Omega)$.

We say that $u\in W^{1,p}(\Omega)$ 
is a weak supersolution of 
\begin{equation}\label{equation_name}
\Delta_{p}u=f,\quad  \mbox{in}\:\: \Omega,
\end{equation}
  if  for every  $\varphi \in
C_0^{\infty}(\Omega),$ $\varphi\geq 0$, there holds that
$$
-\int_{\Omega} |\nabla u(x)|^{p-2}\nabla u \cdot \nabla
\varphi\, dx \leq \int_{\Omega} \varphi\, f\, dx.
$$
Analogously, we say that $u\in W^{1,p}(\Omega)$ 
is a weak subsolution of 
\eqref{equation_name},
  if  for every  $\varphi \in
C_0^{\infty}(\Omega),$ $\varphi\geq 0$, there holds that
$$
-\int_{\Omega} |\nabla u(x)|^{p-2}\nabla u \cdot \nabla
\varphi\, dx \geq \int_{\Omega} \varphi\, f\, dx.
$$
Finally,  $u\in W^{1,p}(\Omega)$ is a weak solution to \eqref{equation_name} if it is both a weak sub- and supersolution.
\end{defn}

Next we recall the following standard notion.
\begin{defn}Given $u, v \in C(\Omega)$, we say that $v$
touches $u$ by below (resp. above) at $x_0 \in \Omega$ if $u(x_0)=
v(x_0)$ and
$$u(x) \geq v(x) \quad (\text{resp. $u(x) \leq
v(x)$}) \quad \text{in a neighborhood $O$ of $x_0$.}$$ If
this inequality is strict in $O \setminus \{x_0\}$, we say that
$v$ touches $u$ strictly by below (resp. above).
\end{defn}

\begin{defn}\label{defnhsolviscosity} Let $u$ and $f$ be  continuous functions in
$\Omega$. We say that $u$ is a viscosity solution to $ \Delta_{p} u = f$   in
$\Omega$, if the following conditions are satisfied:
\begin{enumerate}
\item[(a)] $u$ is a viscosity subsolution in $\Omega$. That is, for every $v\in C^2(\Omega),$  if $v$ touches $u$ from above at $x_0\in \Omega$ and $\nabla v(x_0)\not=0,$ then $\Delta_{p}v(x_0)\geq f(x_0)$.
\item[(b)] $u$ is a viscosity supersolution in $\Omega$. That is, for every $v\in C^2(\Omega),$  if $v$ touches $u$ from below at $x_0\in \Omega$ and $\nabla v(x_0)\not=0,$ then $\Delta_{p}v(x_0)\leq f(x_0)$.
\end{enumerate}
\end{defn}

\begin{rem} \label{equiv-not}
We point out that in \cite{JJ} and \cite{MO} it was shown the equivalence between weak and viscosity subsolutions, supersolutions and solutions to the $p$-Laplacian with right hand side $f\not\equiv 0$.  See \cite{JLM} 
for the case $f\equiv 0$ as well.
\end{rem}

In order to introduce the notion of viscosity solution to \eqref{fbtrue} we assume that 
$$G(t,x):[0,\infty)\times\Omega\rightarrow(0,\infty),$$
is continuous and strictly increasing.

We denote
$$t^+ = \max\{t,0\}, \quad t^-= -\min\{t,0\}.$$

Along the present paper we will use the following definition of viscosity solution to problem \eqref{fbtrue}.

\begin{defn}\label{defnhsol1bis} Let $u$ be a continuous function in
$\Omega$ and $f$ as in \eqref{assump-f}. We say that $u$ is a viscosity solution to (\ref{fbtrue}) in
$\Omega$ if:
\begin{enumerate}
\item $u$ satisfies $ \Delta_{p} u = f$ in $\Omega^+(u) \cup \Omega^-(u)$ in the viscosity sense.
\item  $u$ satisfies the free boundary condition 
$$
u_\nu^+=G(u_\nu^-, x), \quad \mbox{on}\quad F(u)
$$
in the following sense. Let $x_0 \in F(u)$, then for any unit vector $\nu$, there exists no function $\psi\in C^2$ defined in a neighborhood of $x_0$, with $\psi(x_0)=0$, $\nabla \psi (x_0)=\nu$, such that either of the following holds:
\begin{itemize}
\item[(1)] $a\psi^+-b\psi^-\leq u$, with $a>0$, $b>0$ and $a>G(b, x_0)$ (i.e., $u$ is a supersolution); 
\item[(2)] $a\psi^+-b\psi^-\geq u$, with $a>0$, $b>0$ and $a<G(b, x_0)$ (i.e., $u$ is a subsolution).
\end{itemize}
\end{enumerate}
\end{defn}

\begin{rem}
\label{equivalence_free_b_cond}
Condition (ii)-(1) in Definition \ref{defnhsol1bis} is equivalent to saying that we cannot touch by below $u$ at $x_0\in F(u)$ by a piecewise $C^2$ function 
$$
\varphi\in C^2(\{\varphi\geq 0\})\cap C^2(\{\varphi\leq 0\}),\:\:\mbox{with level surface}\:\: \{\varphi=0\}\:\: \mbox{of class}\:\: C^2 
$$
that satisfies 
\begin{equation}\label{freeboundaryassumption}
|\nabla \varphi^+(x_0)|>G(|\nabla \varphi^-(x_0)|, x_0)\:\:\mbox{and}\:\:|\nabla \varphi^-(x_0)|>0.
\end{equation}

A similar statement holds for condition (ii)-(2), for a function $\varphi$ touching $u$ by above at $x_0\in F(u)$, 
 if we replace \eqref{freeboundaryassumption} by 
\begin{equation*}
|\nabla \varphi^+(x_0)|<G(|\nabla \varphi^-(x_0)|, x_0)\:\:\mbox{and}\:\:|\nabla \varphi^+(x_0)|>0.
\end{equation*}
\end{rem}

\bigskip

Now we recall the definition of viscosity solution to problem (\ref{fbtrue}) that we use in \cite{FL3}.

\begin{defn}[Definition 7.1 in \cite{FL3}] \label{defnhsol1} Let $u$ be a continuous function in
$\Omega$ and $f$ as in \eqref{assump-f}. We say that $u$ is a viscosity solution to (\ref{fbtrue}) in
$\Omega$, if the following conditions are satisfied:
\begin{enumerate}
\item $ \Delta_{p} u = f$ in $\Omega^+(u) \cup \Omega^-(u)$ in the weak sense of Definition \ref{defnweak}.
\item  Let $x_0 \in F(u)$ and $v \in C^2(\overline{B^+(v)}) \cap C^2(\overline{B^-(v)})$ ($B=B_\delta(x_0)$) with $F(v) \in C^2$. If $v$ touches $u$ by below (resp. above) at $x_0 \in F(v)$, then 
$$
v_\nu^+(x_0) \leq G(v_\nu^-(x_0), x_0) \quad (\text{resp. $ \geq$)}.
$$
\end{enumerate}
\end{defn}

\smallskip

Next theorem follows as a consequence of Remark \ref{equiv-not}.

\begin{thm}\label{defnhsol2} Let $u$ be a continuous function in $\Omega$ and $f$ as in \eqref{assump-f}. Then $u$ is a viscosity solution to (\ref{fbtrue}) in
$\Omega$, in the sense of Definition \ref{defnhsol1}, if and only if the following conditions are satisfied:
\begin{enumerate}
\item $ \Delta_{p} u = f$ in $\Omega^+(u) \cup \Omega^-(u)$ in the
viscosity sense.
\item  Let $x_0 \in F(u)$ and $v \in C^2(\overline{B^+(v)}) \cap C^2(\overline{B^-(v)})$ ($B=B_\delta(x_0)$) with $F(v) \in C^2$. If $v$ touches $u$ by below (resp. above) at $x_0 \in F(v)$, then 
$$v_\nu^+(x_0) \leq G(v_\nu^-(x_0), x_0) \quad (\text{resp. $ \geq$)}.$$
\end{enumerate}
\end{thm}

\begin{rem}\label{rmk9}
In Definition \ref{defnhsol1bis} (see Remark \ref{equivalence_free_b_cond}) we are requiring that $u$ be tested against functions $\varphi$  such that  $\nabla\varphi^-\not=0$ on $F(u),$ while in Definition \ref{defnhsol1} we are requiring that $u$ be tested against a larger set of functions $\varphi$, since no assumptions on $\nabla \varphi^-$ on $F(u)$ is made.
\end{rem}

\medskip

We will need, in addition,

\begin{defn}
\label{defsubcv} We say that $v \in C(\Omega)$ is a  (strict)
comparison subsolution (resp. supersolution) to \eqref{fbtrue} in $\Omega$,
if  $v\in C^2(\overline{\Omega^+(v)}) \cap C^2(\overline{\Omega^-(v)})$,   $\nabla v\not =0$ in $\overline{\Omega^+(v)}\cup \overline{\Omega^-(v)}$ and the following conditions are satisfied:
\begin{enumerate}
\item $\Delta_{p}v> f $ (resp. $< f $) in $\Omega^+(v) \cup \Omega^-(v)$ (see Remark \ref{comparison-sub-sup}). 
\item If $x_0 \in F(v)$, then
\begin{equation*}
v_\nu^+(x_0)>G(v_\nu^-(x_0),x_0) \quad (\text{resp. $v_\nu^+(x_0)<G(v_\nu^-(x_0),x_0)$).}
\end{equation*}
\end{enumerate}
\end{defn}

 \begin{rem}\label{comparison-sub-sup}
Let $v$ be as in Definition \ref{defsubcv}. Since $v \in C^2(\Omega^+(v)\cup {\Omega^-(v) })$ and  $\nabla v\not =0$ in $\Omega^+(v)\cup {\Omega^-(v) }$  then  $ \Delta_{p} v  > f $ (resp. $<f$) in  $\Omega^+(v)\cup {\Omega^-(v) }$ is understood  pointwise, in the sense of \eqref{equation_nondivergence_form}.  
\end{rem}

\begin{rem}\label{fb-C2}
Notice that by the implicit function theorem, according to our definition, the free boundary of a comparison sub/supersolution is $C^2$.
\end{rem}
\begin{rem} \label{comp-not-touch}
Any (strict) comparison subsolution $v$ (resp. supersolution) cannot touch a viscosity solution $u$ by below (resp. by above) at a point $x_0 \in F(v) $ (resp. $F(u)$).
\end{rem}

\smallskip

\noindent {\bf Notation.}  From now on $B_{\rho}(x_0)\subset {\R}^n$ will denote the open ball of radius $\rho$ centered at $x_0$, and 
$B_{\rho}=B_{\rho}(0)$. Unless otherwise stated, a positive constant depending only on the dimension $n$ and $p$
 will be called universal.

 We will use $c$, $c_i$ to denote small universal constants and $C$, $C_i$ to denote large universal constants.

\section{Regularity of flat free boundaries}\label{section3}

In this section  we  refine a $C^{1,\gamma}$ regularity result for flat free boundaries we proved in  \cite{FL3}, obtaining Theorem \ref{main_new-nondeg-general}. 
We develop, in addition, some improvements of it for the particular cases $f\equiv 0$, $p=2$ and $G(t,x)\equiv G(t)$. 
See Theorems \ref{main_new-nondeg-general-rhs0}, \ref{main_new-nondeg-general-p=2} and \ref{main_new-nondeg-general-G(t)} ---which are  key tools in the proof of Theorem \ref{Lipschitz_cont}. Moreover, we obtain some preliminary results, based on these theorems, that play a fundamental role in the proof of Theorem \ref{Lipschitz_cont} (see Remark \ref{further_remark}).

In fact, in \cite{FL3} we studied problem \eqref{fbtrue} with $f$ satisfying \eqref{assump-f}. We assumed the following basic hypotheses on  the function $G$:
$$G(\eta,x):[0,\infty)\times\Omega\rightarrow(0,\infty)$$
and, for  $0<\hat{\beta}<L$,
\begin{itemize}
\item[(H1)]
$G(\eta,\cdot)\in C^{0,\bar\gamma}(\Omega)$ uniformly in $\eta\in [\frac{\hat{\beta}}{2},4L]$; $G(\cdot,x)\in C^{1,\bar\gamma}([\frac{\hat{\beta}}{2},4L])$ for every $x\in \Omega$ and $G\in L^\infty((\frac{\hat{\beta}}{2},4L)\times \Omega)$.
\item[(H2)]
$G'(\cdot, x)>0$ in $[\frac{\hat{\beta}}{2},4L]$ for $x\in \Omega$ and, for some $\gamma_0$  constant, $G\geq\gamma_0>0$ in $[\frac{\hat{\beta}}{2},4L]\times \Omega$.
\end{itemize}
These assumptions are complemented with the following structural conditions:
\begin{itemize}
\item[(H3)]
There exists $C>0$ such that
 $0\leq G'' (\cdot, x)\leq C$ in $[\frac{\hat{\beta}}{2},4L]$ for $x\in \Omega$. 
 \item[(H4)] There exists $\delta>0$ such that $$G(\eta, x)\geq \eta \frac{\partial G}{\partial \eta}(\eta, x) +\delta, \quad\mbox{for all}\:\: \frac{\hat{\beta}}{2}\leq \eta\leq 4L\:\: \mbox{and}\:\: x\in \Omega.$$
\end{itemize}

\medskip

In this section we  include improvements of our results in \cite{FL3}, obtained under the alternative assumptions on  the function $G$:
\begin{itemize}
\item[(H3')]
There exists $N>0$ such that $\eta^{-N}G(\eta,x)$ is  decreasing in $\eta\in [{\hat{\beta}},2L]$, uniformly in $x\in \Omega$.
\end{itemize}
or otherwise
\begin{itemize}
\item[(H3'')]
$G(t,x)\equiv G(t)$, with $\eta^{-1}G(\eta)$ strictly   decreasing in $\eta\in [{\hat{\beta}},2L]$.
\end{itemize}

\bigskip

Now, let us denote $U_\beta$ the one-dimensional function, $$U_\beta(t) = \alpha t^+ - \beta t^-, \quad  \beta \geq 0, \quad \alpha = G(\beta,0)=G_0(\beta),$$ where $$t^+ = \max\{t,0\}, \quad t^-= -\min\{t,0\}.$$

 Then $U_\beta(x)= U_\beta(x_n)$ is the so-called two-plane solution to \eqref{fbtrue} when $f \equiv 0$ and $G(\eta,x)\equiv G(\eta,0)=G_0(\eta)$.

\smallskip

As a consequence of our results in \cite{FL3}, the following theorem holds.

\begin{thm}[Flatness implies $C^{1,\gamma}$] \label{main_new-nondeg-general} Let $u$ be a viscosity solution to  \eqref{fbtrue} in $B_1$ in the sense of Definition \ref{defnhsol1bis}. Let $0<\hat{\beta}<L$.  Assume $f\in L^\infty(B_1)$ is continuous in $B_1^{+}(u)\cup B_1^{-}(u)$ and $G$ satisfies assumptions (H1)-(H2)-(H3)-(H4) in $B_1$. There exists a universal
constant $\bar \ep>0$ such that, if \begin{equation*}\|u - U_{\beta}\|_{L^{\infty}(B_{1})} \leq \bar \eps\quad 
\text{for some }\, 0 <\hat{\beta}\leq \beta \leq L,\end{equation*} 
 and $$\|f\|_{L^\infty(B_1)} \leq \bar \ep,$$
$$[G(\eta,\cdot)]_{C^{0,\bar \gamma}(B_1)} \leq \bar \ep
, \quad \text{for all }\, 0<\frac{\hat{\beta}}{2} \leq \eta \leq 4L,$$
then 
$$F(u) \text{ is } C^{1,\gamma} \text{ in } B_{1/2} \quad and \quad u\in C^{1,\gamma}(\overline{B^+_{1/2}(u)})\cap 
C^{1,\gamma}(\overline{B^-_{1/2}(u)}).$$ 
Here $\gamma$ is universal and the  $C^{1,\gamma}$ norm of $F(u)$ and of $u$ in $\overline{B^+_{1/2}(u)}$ and $\overline{B^-_{1/2}(u)}$ are bounded by a universal constant.

In the present theorem a constant depending only on $n$, $p$, $\hat\beta$, $L$,  $[G(\eta,\cdot)]_{C^{0,\bar\gamma}(B_1)}$, $\|G(\cdot, x)\|_{C^{1,\bar\gamma}([\frac{\hat{\beta}}{2},4L])}$, $\|G\|_{L^\infty((\frac{\hat{\beta}}{2},4L)\times B_1)}$, $\gamma_0$ and the constants $C$  and $\delta$ in assumptions (H3)-(H4)  is called universal.
\end{thm}
\begin{proof}
 In fact, in the hypotheses of the present theorem,  we showed in Theorem 1.3 in \cite{FL3} that,  if $u$ is a viscosity solution to \eqref{fbtrue} in the sense of Definition  \ref{defnhsol1}, then $F(u)$ is $C^{1,\gamma}$ in $B_{1/2}$. A careful analysis of the proofs in \cite{FL3} shows that Theorem 1.3 in \cite{FL3} holds even when $u$ is a viscosity solution in the sense of Definition \ref{defnhsol1bis} (see Remark  \ref{rmk9}).  

The $C^{1,\gamma}$ regularity of $u$ up to $F(u)$ is an immediate consequence (see,  \cite{Li}, Theorem 1 or \cite{Fan}, Theorem 1.2). This concludes the present proof.
\end{proof}
\begin{rem}\label{capitalremark}
We point out that the same arguments in the proof of Theorem \ref{main_new-nondeg-general} apply to all the theorems in \cite{FL3} as well. Namely,  stronger versions of all the theorems in \cite{FL3} hold.
\end{rem}

\begin{rem}
Let  $G(t,x)=(g(x)+t^{p})^{1/p}$ be such that  $g\in C^{0, \bar{\gamma}}(B_1),$ for some $0<\bar{\gamma}<1$, with $0< g_0\leq g(x)\leq g_1 $ for some constants $g_0, g_1$. Then Theorem \ref{main_new-nondeg-general} 
applies (see Remark 7.9 in \cite{FL3}). 
\end{rem}
\begin{rem}
Other interesting examples of functions $G$ satisfying assumptions (H1)-(H2)-(H3)-(H4) can be found in Section 7 of \cite{FL3}.
\end{rem}

\smallskip

We will next show that in our results in \cite{FL3} we can consider a more general class of functions $G$ giving the  free boundary condition, when we have $f\equiv 0$.
More precisely, we will next consider problem
\begin{equation}  \label{fbtruef0}
\left\{
\begin{array}{ll}
\Delta_{p} u = 0, & \hbox{in $\Omega^+(u)\cup \Omega^-(u)$}, \\
\  &  \\
u_\nu^+=G(u_\nu^-,x), & \hbox{on $F(u):= \partial \Omega^+(u) \cap
\Omega.$} 
\end{array}
\right.
\end{equation}

\smallskip

We will prove

\begin{thm}[Flatness implies $C^{1,\gamma}$] \label{main_new-nondeg-general-rhs0} Let $u$ be a viscosity solution to  \eqref{fbtruef0} in $B_1$ in the sense of Definition \ref{defnhsol1bis}. Let $0<\hat{\beta}<L$.  Assume $G$ satisfies  (H1)-(H2)-(H3') in $B_1$. There exists a universal
constant $\bar \ep>0$ such that, if \begin{equation*}\|u - U_{\beta}\|_{L^{\infty}(B_{1})} \leq \bar \eps\quad 
\text{for some }\, 0 <\hat{\beta}\leq \beta \leq L,\end{equation*} 
 and 
$$[G(\eta,\cdot)]_{C^{0,\bar \gamma}(B_1)} \leq \bar \ep
, \quad \text{for all }\, 0<\frac{\hat{\beta}}{2} \leq \eta \leq 4L,$$
then 
$$F(u) \text{ is } C^{1,\gamma} \text{ in } B_{1/2} \quad and \quad u\in C^{1,\gamma}(\overline{B^+_{1/2}(u)})\cap 
C^{1,\gamma}(\overline{B^-_{1/2}(u)}).$$ 
Here $\gamma$ is universal and the  $C^{1,\gamma}$ norm of $F(u)$ and of $u$ in $\overline{B^+_{1/2}(u)}$ and $\overline{B^-_{1/2}(u)}$ are bounded by a universal constant.

In the present theorem a constant depending only on $n$, $p$, $\hat\beta$, $L$,  $[G(\eta,\cdot)]_{C^{0,\bar\gamma}(B_1)}$, $\|G(\cdot, x)\|_{C^{1,\bar\gamma}([\frac{\hat{\beta}}{2},4L])}$, $\|G\|_{L^\infty((\frac{\hat{\beta}}{2},4L)\times B_1)}$, $\gamma_0$ and the constant $N$ in assumption (H3') is called universal.
\end{thm}

In order to obtain Theorem \ref{main_new-nondeg-general-rhs0} we will need

\begin{lem}
\label{mainvar} Assume $G$ satisfies (H1)-(H2)-(H3') in $B_1$. There exists
a universal constant $\bar \ep>0$ such that if $u$ is a viscosity solution of $\eqref{fbtruef0}$  in the sense of Definition \ref{defnhsol1} and  satisfies
\begin{equation}\label{flat-var}
U_\beta(x_n+\sigma)\leq u(x)\leq U_\beta(x_n+\sigma+\varepsilon),\quad x\in B_1,  \ \ |\sigma|<\frac{1}{20},
\end{equation} 
for some $0 < \hat{\beta}\leq\beta \leq L$, with 
\begin{equation}  \label{noncvar}
\|G(\eta,x) - G_0(\eta)\|_{L^\infty(B_1)}\leq \varepsilon^2, \quad \text{ for all } \hat{\beta}\leq \eta \leq 2L,
\end{equation}
and in $\bar x=\frac{1}{10}e_n,$ 
$$
u(\bar x)\geq U_\beta(\bar x_n+\sigma+\frac{\varepsilon}{2}),
$$
for some $\ep\le\bar{\ep}$, then
\begin{equation}\label{improv-below}
 u(x) \ge U_\beta(x_n+\sigma+c{\varepsilon}) \quad \text{in } \overline{B}_{\frac{1}{2}},
\end{equation}
for some universal $0<c<1.$ Analogously, if 
$$
u(\bar x)\leq U_\beta(\bar x_n+\sigma+\frac{\varepsilon}{2}),
$$
then
\begin{equation}
 u(x) \le U_\beta(x_n+\sigma+(1-c){\varepsilon}) \quad \text{in } \overline{B}_{\frac{1}{2}}.
\end{equation}
\end{lem}
\begin{proof}We will argue as in the proof of Lemma 7.4 in \cite{FL3} and we will also apply arguments similar to those in Lemma 8.1 in \cite{DFS1}. 

We prove the first statement and for notational simplicity we drop the sub-index $\beta$ from $U_\beta$.

We start by observing that  we have
\begin{equation*}
|\nabla u|^{p-2}\mathcal{M}_{\lambda_0,\Lambda_0}^-(D^2u) 
\leq \Delta_{p}u = 0, \quad \text{ in }B_1^+(u)\cup B_1^-(u),
\end{equation*}
in the viscosity sense, where $\lambda_0:=\min\{1,p-1\}$ and $\Lambda_0:=\max\{1,p-1\}.$ 
Hence, applying Lemma 6 in \cite{IS}, we conclude that 
\begin{equation}\label{pucci-minus}
\mathcal{M}^-_{\lambda_0,\Lambda_0}(D^2u)\leq 0, \quad \text{ in }B_1^+(u)\cup B_1^-(u),
\end{equation}
in the viscosity sense.

{}From \eqref{flat-var} we have that $u(x)\ge U(x_n+\sigma)$ in $B_1$ and  that $B_{1/20}(\bar{x})\subset B_1^+(u).$ Then, 
\begin{equation*}
\Delta_{p}u=0\quad \text{ in } B_{1/20}(\bar{x}).
\end{equation*}
 Thus,  $u\in C^{1,{\widetilde{\gamma}}}$ in $\overline{B}_{1/40}(\bar{x}),$  where ${\widetilde{\gamma}}\in (0,1)$ and
$||u||_{C^{1,{\widetilde{\gamma}}}(\overline{B}_{1/40}(\bar{x}))}\leq C,$ with $C\geq 1$. Here ${\widetilde{\gamma}}$ and $C$ are universal constants depending
only on $p,n,L$ and $G_0(L)$. We have used  that \eqref{flat-var} implies that 
$||u||_{L^{\infty}(B_1)}\le 2\max\{L, G_0(L)\}$.

We consider two cases:

\medskip

{\bf Case (i).} Suppose $|\nabla u(\bar{x})|<\frac{\alpha}{4}$. As in Lemma 7.4 in \cite{FL3} we denote $q(x)=\alpha(x_n+ \sigma)$ and 
 obtain  
\begin{equation}\label{bound-around-barx0}
\alpha\frac{c_0}{2}\varepsilon\leq u(x)-q(x) \quad \text{ in }B_{r_3}(\bar{x}_0),
\end{equation}
with $\bar{x}_0:=\bar{x}-r_2e_n$. The constants $c_0, r_2, r_3$ are universal,  depending only on
$p,n,L, G_0(L)$ and $\gamma_0$.

{}From \eqref{bound-around-barx0} we will deduce that for $c_1>0$ small universal
\begin{equation}\label{c11}u -\alpha(x_n+ \sigma) \geq \alpha c_1 \eps (x_n+ \sigma), \quad x \in \{x_n>-\sigma\} \cap \overline B_{19/20}.\ee

To prove this claim, let $\phi$ solve
$$\mathcal{M}_{\lambda_0,\Lambda_0}^-(D^2 \phi) = 0 \quad \text{in $R:= (B_1 \cap \{x_n >-\sigma\})\setminus \overline B_{r_3}(\bar{x}_0)$}$$ with boundary data
$$\phi=0 \quad \text{on $\p (B_1 \cap \{x_n >-\sigma\})$}, \quad \phi=1 \quad \text{on $\p  B_{r_3}(\bar{x}_0).$}$$ 
Then, by boundary Harnack (see, for instance Proposition 2.5 in \cite{SS})
$$\phi \geq \hat{c} (x_n+ \sigma) \quad \text{in $\bar R \cap B_{19/20},$}$$
with $\hat{c}$ universal.

Now we observe that \eqref{flat-var} and \eqref{pucci-minus} imply that
$$\mathcal{M}^-_{\lambda_0,\Lambda_0}(D^2u)\leq 0\quad\text{ in }\quad B_1\cap\{x_n >-\sigma\}.$$

As a consequence
$$\mathcal{M}^-_{\lambda_0,\Lambda_0}(D^2(u-\alpha(x_n+ \sigma) ))\leq 0=\mathcal{M}^-_{\lambda_0,\Lambda_0}(D^2(\frac{1}{2}\alpha c_0\phi \varepsilon))\quad\text{ in } R.$$
So, recalling \eqref{bound-around-barx0} and using \eqref{flat-var} again, we get
$$u-\alpha(x_n+ \sigma)\ge \frac{1}{2}\alpha c_0\phi \varepsilon \quad\text{ in } R.$$
Hence \eqref{c11} follows.

Let us  define the function $w:\bar{A}\to\mathbb{R},$ $A:=B_{\frac{4}{5}}(\bar{x}_0)\setminus \bar{B}_{r_3}(\bar{x}_0)$ as
$$
w(x)=\bar{c}\left(|x-\bar{x}_0|^{-\gamma}-(\frac{4}{5})^{-\gamma}\right),
$$
for $\gamma=\gamma(n,p)\ge 1$ given in Theorem 4.4 in \cite{FL3}. We choose $\bar{c}>0$  in such a way that
$$
w=\left\{\begin{array}{l}
0,\quad \mbox{on}\quad \partial B_{\frac{4}{5}}(\bar{x}_0)\\
1,\quad \mbox{on}\quad \partial B_{r_3}(\bar{x}_0)
\end{array}
\right.
$$
and we extend $w$ to $1$ in $B_{r_3}(\bar{x}_0)$.

Now set $\psi =1-w$ and,  for $x \in
\overline B_{4/5}(\bar{x}_0)
$ and $t\in \R$, define
\begin{equation*}
v_t(x)= \alpha(1+c_1 \eps)(x_n +\sigma - \frac{\varepsilon}{2} c_0 \delta \psi(x)+t\varepsilon)^+ -  \beta(x_n +\sigma- \frac{\varepsilon}{2} c_0 \delta \psi(x)+t\varepsilon)^-, \end{equation*} with $\delta>0$ small universal to be made precise later, and $c_1$ the constant in \eqref{c11}.

Then, for $t=-c_1$ one can  verify that $$v_{-c_1}(x) \leq U(x_n+\sigma) \leq u (x), \quad  x \in
\overline B_{\frac{4}{5}}(\bar{x}_0).$$

Let $\bar t$ be the largest $t \geq -c_1$ such that
\begin{equation*}
v_{t}(x) \leq u(x) \quad \text{in $\overline B_{\frac{4}{5}}(\bar{x}_0)$},
\end{equation*}
and let $\tilde x$ be the first touching point.

We want to show that $\bar t \geq \frac{c_0}{2}\delta.$ Then, arguing as in \cite{FL3}, we will get \eqref{improv-below} for a universal
constant $0<c<1$ depending only on $p,n,L, G_0(L), \delta$ and $\gamma_0$.

Now suppose $\bar t <\frac{c_0}{2}\delta.$  To guarantee that $\tilde x$ cannot belong to $\p  B_{\frac{4}{5}}(\bar{x}_0)$ when 
$\bar t < \frac{c_0}{2}\delta$ we use \eqref{c11}. Indeed if $x \in \p B_{\frac{4}{5}}(\bar{x}_0) $ and $v_{\bar t}(x)\geq 0$ then, in view of \eqref{c11},
$$v_{\bar t}(x)= \alpha(1+c_1 \eps)(x_n +\sigma-\eps \frac{c_0}{2}\delta + \bar t \eps) <  \alpha(1+c_1 \eps)(x_n +\sigma)\leq u(x).$$ If $v_{\bar t}(x) <0$ we use that $u(x) \geq U(x_n+\sigma)$ to reach again the conclusion that $v_{\bar t}(x) < u(x)$.

We will now show that $\tilde x$ cannot belong to the annulus $A$. Indeed, we will get a contradiction, if we show that $v_{\bar{t}}$ is a strict
subsolution to $\eqref{fbtruef0}$ in $A$. 

In fact,  we obtain from Theorem 4.4 in \cite{FL3} that 
\begin{equation*}
\Delta_{p}v_{\bar t} > 0 \quad \text{ in } A^+(v_{\bar t})\cup A^-(v_{\bar t})
\end{equation*}
if $\varepsilon\le\ep_1$, with $\ep_1=\ep_1(n,p, L, G_0(L), \gamma_0)$.  Also, as in Lemma 4.5 in \cite{FL3} (see (4.18)),
\begin{equation*}
-\tilde{c}_1\le\psi_n \le-\tilde{c}_2 <0 \quad \text{on $F(v_{\bar t}) \cap A$},
\end{equation*}
with $\tilde{c}_1$ and $\tilde{c}_2$  universal constants, for $\ep\le \ep_2$, with $\ep_2$ universal. Then 
we have
\begin{equation*}
k \equiv|e_n-\varepsilon \frac{c_0}{2}\delta\nabla\psi|=(1-\varepsilon c_0\delta 
\psi_n+\varepsilon^2\frac{c_0^2}{4}\delta^2|\nabla\psi|^2)^{1/2} = 1+ \tilde k  \delta\eps,
\end{equation*} 
where $0<c_1\leq\tilde{k}\leq c_2$, with $c_1,c_2$ universal  constants and moreover,  
\begin{equation}\label{controlofkappa}1< k \leq 2,
\end{equation}
if $\varepsilon\leq \varepsilon_3$ universal. We will show that, on $F(v_{\bar t}) \cap A,$ using \eqref{noncvar},  we have
\begin{equation*}
(v_{\bar{t}}^{+})_{\nu }-G((v_{\bar{t}}^{-})_{\nu },x)>0,
\end{equation*}
as long as $\ep\le \ep_4$ universal. In fact, recalling  \eqref{noncvar} and \eqref{controlofkappa},
we get
\begin{align*}
(v_{\bar{t}}^{+})_{\nu }-G((v_{\bar{t}}^{-})_{\nu },x)&=
\alpha (1+c_1\varepsilon) k -G(\beta k ,x) \geq \alpha (1+c_1\varepsilon
)k -G_0(\beta k)-\varepsilon^2\\
&\ge (1+c_1 \varepsilon)k G_0(\beta)-G_0(\beta)k^N-\varepsilon^2
\\
&\geq \varepsilon G_0(\beta)(\frac{c_1}{2}-2^N N\tilde k \delta)> 0
\end{align*}
if $\delta< c_1/(2^{N+1}N\tilde k)$ and $\varepsilon\le \varepsilon_5$ universal.
We used that $G_0(\beta)\geq \gamma_0>0$ and that $G_0(\beta k)\le G_0(\beta)k^N$, since $\eta ^{-N}G_{0}(\eta )$ is  decreasing.

 Thus, $v_{\bar t}$ is a
strict subsolution to \eqref{fbtruef0} in $A$ as desired, which shows that $\tilde x$ cannot belong to the annulus $A$.

\smallskip

Therefore, $\tilde x \in \overline {B}_{r_3}(\bar{x}_0)$ and
$$u(\tilde x)=v_{\bar t}(\tilde x)  = \alpha (1+c_1\varepsilon)(\tilde x_n+\sigma + \bar t \ep) < \alpha (1+c_1\varepsilon)(\tilde x_n+\sigma + \frac{c_0}{2}\delta \ep).$$ 
This contradicts \eqref{bound-around-barx0}, if we choose $c_1<\frac{c_0}{4}$, $\delta<\frac{1}{4}$, and $\varepsilon\le \varepsilon_6$ universal.

\smallskip

{\bf Case (ii).} Now suppose $|\nabla u(\bar{x})|\geq\frac{\alpha}{4}.$ By exploiting  the $C^{1,{\widetilde{\gamma}}}$ regularity of $u$ in $\overline{B}_{\frac{1}{40}}(\bar{x})$, we know that $u$ is Lipschitz continuous in $\overline{B}_{\frac{1}{40}}(\bar{x})$, as well as there exists a  constant $0<r_0$, with $8r_0\le \frac{1}{40}$, and $C>1$, $r_0$ and $C$ depending only on $n, p,L, G_0(L)$ and $\gamma_0$
 such that 
$$
\frac{\gamma_0}{8}\leq |\nabla u|\leq C \quad \text{ in } B_{8r_0}(\bar{x}).
$$ 
We now combine the argument in Case (ii) of Lemma 4.5 in \cite{FL3} with the ones above. This completes the proof.
\end{proof}

We now obtain

\begin{proof}[\bf Proof of Theorem \ref{main_new-nondeg-general-rhs0}]  The proof follows as the one of Theorem 1.3 in \cite{FL3}, if we replace Lemma 7.4 in \cite{FL3} by Lemma \ref{mainvar} above.
\end{proof}

\medskip

We also  show that in our results in \cite{FL3} we can  consider a more general class of functions $G$ giving the free boundary condition, when we have 
$p=2$.
More precisely, we will next consider problem
\begin{equation}  \label{fbtruep=2}
\left\{
\begin{array}{ll}
\Delta u = f, & \hbox{in $\Omega^+(u)\cup \Omega^-(u)$}, \\
\  &  \\
u_\nu^+=G(u_\nu^-,x), & \hbox{on $F(u):= \partial \Omega^+(u) \cap
\Omega.$} 
\end{array}
\right.
\end{equation}

\smallskip

We will show

\begin{thm}[Flatness implies $C^{1,\gamma}$] \label{main_new-nondeg-general-p=2} Let $u$ be a viscosity solution to  \eqref{fbtruep=2} in $B_1$ in the sense of Definition \ref{defnhsol1bis}. Let $0<\hat{\beta}<L$.  Assume $G$ satisfies  (H1)-(H2)-(H3') in $B_1$. There exists a universal
constant $\bar \ep>0$ such that, if \begin{equation*}\|u - U_{\beta}\|_{L^{\infty}(B_{1})} \leq \bar \eps\quad 
\text{for some }\, 0 <\hat{\beta}\leq \beta \leq L,\end{equation*} 
 and 
$$ \|f\|_{L^\infty(B_1)} \leq \bar \ep,$$
$$[G(\eta,\cdot)]_{C^{0,\bar \gamma}(B_1)} \leq \bar \ep
, \quad \text{for all }\, 0<\frac{\hat{\beta}}{2} \leq \eta \leq 4L,$$
then 
$$F(u) \text{ is } C^{1,\gamma} \text{ in } B_{1/2} \quad and \quad u\in C^{1,\gamma}(\overline{B^+_{1/2}(u)})\cap 
C^{1,\gamma}(\overline{B^-_{1/2}(u)}).$$ 
Here $\gamma$ is universal and the  $C^{1,\gamma}$ norm of $F(u)$ and of $u$ in $\overline{B^+_{1/2}(u)}$ and $\overline{B^-_{1/2}(u)}$ are bounded by a universal constant.

In the present theorem a constant depending only on $n$,  $\hat\beta$, $L$,  $[G(\eta,\cdot)]_{C^{0,\bar\gamma}(B_1)}$, $\|G(\cdot, x)\|_{C^{1,\bar\gamma}([\frac{\hat{\beta}}{2},4L])}$, $\|G\|_{L^\infty((\frac{\hat{\beta}}{2},4L)\times B_1)}$, $\gamma_0$ and the constant $N$ in assumption (H3') is called universal.
\end{thm}

In order to obtain Theorem \ref{main_new-nondeg-general-p=2} we will need

\begin{lem}
\label{mainvar-p=2} Assume $G$ satisfies (H1)-(H2)-(H3') in $B_1$. There exists
a universal constant $\bar \ep>0$ such that if $u$ is a viscosity solution of $\eqref{fbtruep=2}$  in the sense of Definition \ref{defnhsol1} and  satisfies
\begin{equation*}
U_\beta(x_n+\sigma)\leq u(x)\leq U_\beta(x_n+\sigma+\varepsilon),\quad x\in B_1,  \ \ |\sigma|<\frac{1}{20},
\end{equation*} 
for some $0 < \hat{\beta}\leq\beta \leq L$, with 
\begin{equation}\label{assump-f-p=2}
\|f\|_{L^\infty(B_1)} \leq \ep^2\min\{ 1, \beta, G_0(\beta)\}, 
\end{equation}
\begin{equation*}  
\|G(\eta,x) - G_0(\eta)\|_{L^\infty(B_1)}\leq \varepsilon^2, \quad \text{ for all } \hat{\beta}\leq \eta \leq 2L,
\end{equation*}
and in $\bar x=\frac{1}{10}e_n,$ 
$$
u(\bar x)\geq U_\beta(\bar x_n+\sigma+\frac{\varepsilon}{2}),
$$
for some $\ep\le\bar{\ep}$, then
\begin{equation*}
 u(x) \ge U_\beta(x_n+\sigma+c{\varepsilon}) \quad \text{in } \overline{B}_{\frac{1}{2}},
\end{equation*}
for some universal $0<c<1.$ Analogously, if 
$$
u(\bar x)\leq U_\beta(\bar x_n+\sigma+\frac{\varepsilon}{2}),
$$
then
\begin{equation*}
 u(x) \le U_\beta(x_n+\sigma+(1-c){\varepsilon}) \quad \text{in } \overline{B}_{\frac{1}{2}}.
\end{equation*}
\end{lem}
\begin{proof}We will argue as in the proof of Lemma \ref{mainvar}.
In the present case we have
\begin{equation*}
\Delta u = f, \quad \text{ in }B_1^+(u)\cup B_1^-(u).
\end{equation*}
We consider two cases:

\medskip

{\bf Case (i).} Suppose $|\nabla u(\bar{x})|<\frac{\alpha}{4}$. As in Lemma \ref{mainvar} we get that \eqref{bound-around-barx0} holds. We now
let $\phi$ solve
$$\Delta \phi = 0 \quad \text{in $R:= (B_1 \cap \{x_n >-\sigma\})\setminus \overline B_{r_3}(\bar{x}_0)$}$$ with boundary data
$$\phi=0 \quad \text{on $\p (B_1 \cap \{x_n >-\sigma\})$}, \quad \phi=1 \quad \text{on $\p  B_{r_3}(\bar{x}_0).$}$$ 
Then, by boundary Harnack (see, for instance Proposition 2.5 in \cite{SS})
$$\phi \geq \hat{c} (x_n+ \sigma) \quad \text{in $\bar R \cap B_{19/20},$}$$
with $\hat{c}$ universal.

Now let $\tilde\phi =\frac {1}{2}\alpha c_0 \phi \eps - 8\alpha \eps^2 x_n + 4 \alpha \ep^2 x_n^2.$

Using \eqref{assump-f-p=2}, we get
$$u-\alpha(x_n+ \sigma)\ge \tilde\phi \quad\text{ in } R,$$
which implies \eqref{c11}. Now the proof follows exactly as that of Lemma \ref{mainvar}.
\end{proof}

\medskip

We then obtain

\begin{proof}[\bf Proof of Theorem \ref{main_new-nondeg-general-p=2}]  The proof follows as the one of Theorem 1.3 in \cite{FL3}, if we replace Lemma 7.4 in \cite{FL3} by Lemma \ref{mainvar-p=2} above.
\end{proof}

\medskip

We conclude the proofs in this section  by proving that our results in \cite{FL3} can be obtained under a different set of assumptions in the particular case in which the function giving the free boundary condition satisfies that $G(t,x)\equiv G(t)$.
More precisely, we will next consider problem
\begin{equation}  \label{fbtrueG(t)}
\left\{
\begin{array}{ll}
\Delta_p u = f, & \hbox{in $\Omega^+(u)\cup \Omega^-(u)$}, \\
\  &  \\
u_\nu^+=G(u_\nu^-), & \hbox{on $F(u):= \partial \Omega^+(u) \cap
\Omega.$} 
\end{array}
\right.
\end{equation}

\smallskip

We are in a position of obtaining

\begin{thm}[Flatness implies $C^{1,\gamma}$] \label{main_new-nondeg-general-G(t)} Let $u$ be a viscosity solution to  \eqref{fbtrueG(t)} in $B_1$ in the sense of Definition \ref{defnhsol1bis}. Let $0<\hat{\beta}<L$.  Assume $G$ satisfies  (H1)-(H2)-(H3'') in $B_1$. There exists a universal
constant $\bar \ep>0$ such that, if \begin{equation*}\|u - U_{\beta}\|_{L^{\infty}(B_{1})} \leq \bar \eps\quad 
\text{for some }\, 0 <\hat{\beta}\leq \beta \leq L,\end{equation*} 
 and 
$$ \|f\|_{L^\infty(B_1)} \leq \bar \ep,$$then 
$$F(u) \text{ is } C^{1,\gamma} \text{ in } B_{1/2} \quad and \quad u\in C^{1,\gamma}(\overline{B^+_{1/2}(u)})\cap 
C^{1,\gamma}(\overline{B^-_{1/2}(u)}).$$ 
Here $\gamma$ is universal and the  $C^{1,\gamma}$ norm of $F(u)$ and of $u$ in $\overline{B^+_{1/2}(u)}$ and $\overline{B^-_{1/2}(u)}$ are bounded by a universal constant.

In the present theorem a constant depending only on $n$, $p$, $\hat\beta$, $L$,   $\|G\|_{C^{1,\bar\gamma}([\frac{\hat{\beta}}{2},4L])}$, $\|G\|_{L^\infty((\frac{\hat{\beta}}{2},4L))}$ and $\gamma_0$  is called universal.
\end{thm}

In order to obtain Theorem \ref{main_new-nondeg-general-G(t)} we will need

\begin{lem}
\label{mainvar-G(t)} Assume $G$ satisfies (H1)-(H2)-(H3'') in $B_1$. There exists
a universal constant $\bar \ep>0$ such that if $u$ is a viscosity solution of $\eqref{fbtrueG(t)}$  in the sense of Definition \ref{defnhsol1} and  satisfies
\begin{equation*}
U_\beta(x_n+\sigma)\leq u(x)\leq U_\beta(x_n+\sigma+\varepsilon),\quad x\in B_1,  \ \ |\sigma|<\frac{1}{20},
\end{equation*} 
for some $0 < \hat{\beta}\leq\beta \leq L$, with 
\begin{equation*}
\|f\|_{L^\infty(B_1)} \leq \ep^2\min\{ 1, \beta^{p-1}, G(\beta)^{p-1}\},\end{equation*}
and in $\bar x=\frac{1}{10}e_n,$ 
$$
u(\bar x)\geq U_\beta(\bar x_n+\sigma+\frac{\varepsilon}{2}),
$$
for some $\ep\le\bar{\ep}$, then
\begin{equation*}
 u(x) \ge U_\beta(x_n+\sigma+c{\varepsilon}) \quad \text{in } \overline{B}_{\frac{1}{2}},
\end{equation*}
for some universal $0<c<1.$ Analogously, if 
$$
u(\bar x)\leq U_\beta(\bar x_n+\sigma+\frac{\varepsilon}{2}),
$$
then
\begin{equation*}
 u(x) \le U_\beta(x_n+\sigma+(1-c){\varepsilon}) \quad \text{in } \overline{B}_{\frac{1}{2}}.
\end{equation*}
\end{lem}
\begin{proof} The proof follows  as that of Lemma 7.4 in \cite{FL3}. In the present case, to obtain that $v_{\bar t}$ is a
strict subsolution to \eqref{fbtrueG(t)} in $A$, we use (H3'') and we immediately get
\begin{equation*}
(v_{\bar t}^+)_\nu-G((v_{\bar t}^-)_\nu)=G(\beta)k-G(\beta k)>0.
\end{equation*}
\end{proof}
\medskip

We finally obtain

\begin{proof}[\bf Proof of Theorem \ref{main_new-nondeg-general-G(t)}]  The proof follows as the one of Theorem 1.3 in \cite{FL3}, if we replace Lemma 7.4 in \cite{FL3} by Lemma \ref{mainvar-G(t)} above.
\end{proof}

\smallskip

We will now deduce some preliminary results to be applied in the proof of Theorem \ref{Lipschitz_cont} 

\begin{rem}\label{further_remark} 
Let us now consider $G(t,x):[0,\infty)\times B_1\rightarrow(0,\infty)$, $G(\cdot, x)\in C^2(0,\infty)$ for $x\in B_1$, satisfying  (P2) and (P3) in $B_1$. Given a sequence $C_k\to \infty$ we now define 
$$\tilde{G}_k(t,x):=\frac{G(C_kt, x)}{C_k},\qquad \mbox{ for }t>0 \mbox{ and }x\in B_1.$$ 

Let us  verify that $\tilde{G}_k$ satisfy (H1), (H2) and (H3') in $B_1$, with $\hat{\beta}=\frac{1}{4}$ and $L=C$ (where $C$ is a given universal constant), with constants independent of  $k$, for $k$ large.

Clearly, $\tilde{G}_k(t,\cdot)\in C^{0,\bar\gamma}(B_1)$ uniformly in $t\in [\frac{1}{8},4C]$, uniformly in $k$.

Due to the second condition in assumption (P2), we have,  for $t\in [\frac{1}{8},4C]$ and $x\in B_1$,
$$\Big|\frac{{\partial}^2\tilde{G}_k}{\partial t^2}(t,x)\Big|=\Big|C_k\frac{{\partial}^2{G}}{\partial t^2}(C_kt,x)\Big|\le \hat{C},$$
for some constant $\hat{C}$, if $k$ is large. This implies that $\tilde{G}_k(\cdot, x)\in C^{1,1}([\frac{1}{8}, 4C])$, for $x\in B_1$, uniformly in $k$, for $k$ large.

On the other hand, since $G$ satisfies assumption (P2) then, for some constant $t_0>0$,  
\begin{equation}\label{Gb}
\frac{1}{2} \le \frac{\partial G}{\partial t}(t,x) \le 2 \qquad\mbox{ for } t\ge t_0 \mbox{ and } x\in B_1.
\end{equation}
Then, for $t\ge t_0$ and $x\in B_1$,
\begin{equation*}
\frac{1}{2}(t-t_0) \le G(t,x) - G(t_0,x) = \int_{t_0}^t \frac{\partial G}{\partial t}(s,x) \,ds \le {2}(t-t_0),
\end{equation*}
and this implies, for $t\ge t_1$ and $x\in B_1$,
\begin{equation}\label{Ga}
\frac{1}{4}t \le G(t,x) \le 2(t-t_0) +\overline{C} \le 3t.
\end{equation}
Also, for $t\ge t_1$, $x\in B_1$ and large $k$,
\begin{equation*}
\frac{1}{4}t \le \tilde{G}_k(t,x)   \le 3t.
\end{equation*}
We deduce that for $t\in [\frac{1}{8}, 4C])$, $x\in B_1$ and $k$ large,
\begin{equation*}
c_1\le \tilde{G}_k(t,x)\le C_1,
\end{equation*}
where $c_1$ and $C_1$ are positive constants that depend only on $C$ and $G$ and not on $k$. We also observe that \eqref{Gb} and \eqref{Ga} yield, for $t\ge t_2$ and $x\in B_1$,
\begin{align*}
& \frac{\partial} {\partial t}(t^{-N}G(t,x))=t^{-N}(-N t^{-1}G(t,x)+\frac{\partial} {\partial t} G(t,x))\\
&\le t^{-N}(-N t^{-1}\frac{1}{4}t +2)\le 0,
\end{align*}
if $N\ge 8$. As a consequence, if $N\ge 8$, $t\ge \frac{1}{4}$, $x\in B_1$  and $k$ large,
\begin{equation*}
 \frac{\partial} {\partial t}(t^{-N}\tilde{G}_k(t,x))\le 0.
\end{equation*}
That is, we have shown that if $G$ satisfies   (P2) and (P3) in $B_1$, then $\tilde{G}_k$ satisfy (H1), (H2) and (H3') in $B_1$, with $\hat{\beta}=\frac{1}{4}$ and $L=C$, with constants independent of $k$, for $k$ large.

Then, if $f\equiv 0$ or $p=2$ we can apply Theorem \ref{main_new-nondeg-general-rhs0} or \ref{main_new-nondeg-general-p=2} for this family of functions, obtaining the conclusion with constants independent 
of $k$, for large $k$.

 If $f \not\equiv 0$ and $p\neq 2$ we assume in addition that $G$ satisfies (P4), then $\tilde{G}_k$ satisfy (H3''), for $t\in [\frac{1}{4}, 2C]$, $x\in B_1$  and $k$ large. 

Therefore, in this situation we can apply Theorem  \ref{main_new-nondeg-general-G(t)} for this family of functions obtaining  the conclusion with constants independent of $k$ as well, for large $k$.
 \end{rem}

\section{H\"older continuity}\label{section4}
In this section we prove the local H\"older continuity of viscosity solutions to problem \eqref{fbtrue}. The result holds for a very general class of problems of this type.

More precisely, we obtain Theorem \ref{holder_reg}, which is obtained under the following assumptions on $G$: we require that there exist constants  $\sigma>0$ and $M>0$ such that
\begin{equation*}
\sigma t\leq G(t, x)\leq \sigma^{-1}t,\quad \mbox{for}\:\:t>M,\: x\in B_1.
\end{equation*}

\smallskip

In the proof of Theorem \ref{holder_reg} we will use the following fundamental lemma.

\begin{lem}\label{weak-harnack-p-lap} Let $0<r\le 1$ and $f\in L^{\infty}(B_r)$ with $||f||_{L^{\infty}(B_r)}\le 1$. Let $u$ be a bounded weak subsolution to $\Delta_p u=f$ in $B_r$, $u\ge 0$. If 
\begin{equation}\label{cond-measure}
\frac{|\{u = 0\} \cap B_{r/2}|}{|B_{r/2}|} \geq \frac 12 \quad
\end{equation}
then,
$$\sup_{B_{r/2}} u \leq (1-\eta)\sup_{B_{r}}u+||f||^{1/p}_{L^{\infty}(B_r)}.$$
Here $0<\eta<1$ is a constant depending only on $n$ and $p$.
\end{lem}
\begin{proof} We apply the weak Harnack inequality for the $p$-Laplace operator (see Theorem 1.2 and Corollary 1.1 in \cite{T}) to $v=M_r-u$, where $M_r=\sup_{B_r} u$. Then, we get, for some positive $C=C(n,p)$ and $\gamma=\gamma(n,p)$,
$$||v+\theta||_{L^{\gamma}(B_{r/2})}\le r^{n/\gamma}C\inf_{B_{r/2}}( v+\theta).$$  
Here
\begin{equation}\label{def-sigma}
\theta=\mu \frac{r}{2} + (\mu \frac{r}{2})^{p/(p-1)},\qquad \mu=||f||^{1/p}_{L^{\infty}(B_r)}.
\end{equation} 
As a consequence,
$$||M_r-u||^{\gamma}_{L^{\gamma}(B_{r/2})}\le C^{\gamma}\left[r^{n}\left(\inf_{B_{r/2}}( M_r-u+\theta)\right)^{\gamma} +||\theta||^{\gamma}_{L^{\gamma}(B_{r/2})}\right].$$  
Now, applying \eqref{cond-measure}, we get
$$\frac{1}{2}{M_r}^{\gamma}|B_{r/2}|\le   C^{\gamma}\left[r^n  \left(M_r-M_{r/2}+\theta\right)^{\gamma}+{\theta}^{\gamma}|B_{r/2}|\right].$$  
Recalling \eqref{def-sigma}, we obtain, for some $\widetilde{C}=\widetilde{C}(n,p)>1$,
$${M_r}\le  \widetilde{C}\left(M_r-M_{r/2}+||f||^{1/p}_{L^{\infty}(B_r)}\right).$$
Finally, denoting $1-\eta=\frac{\widetilde{C}-1}{\widetilde{C}}$, we deduce that
$$\sup_{B_{r/2}} u \leq (1-\eta)\sup_{B_{r}}u+||f||^{1/p}_{L^{\infty}(B_r)},$$
as claimed.
\end{proof}

\smallskip

We can now obtain

\begin{proof}[\bf Proof of Theorem \ref{holder_reg}] Without loss of generality we will assume that $u$ satisfies the statement in $B_2$.

We start by assuming that $0\in F(u)$ (the other case will be discussed at the end of the proof). 

For $C>0$, let us consider, for $x\in B_1$,
\begin{equation}\label{frac-u-C}
\hat u := \frac{u}{C},\quad \hat{f}:=\frac{f}{C^{p-1}}, \quad \hat G(t,x):= \frac{G(C t,x)}{C}.
\end{equation}
Then, $\hat{u}$ is a viscosity solution to
\begin{equation*}  
\left\{
\begin{array}{ll}
\Delta_{p} \hat{u} = \hat{f}, & \hbox{in $B_1^+(\hat{u})\cup B_1^-(\hat{u})$}, \\
\  &  \\
{\hat{u}}_\nu^+=\hat{G}({\hat{u}}_\nu^-,x), & \hbox{on $F(\hat{u})$}. 
\end{array}
\right.
\end{equation*}

We can choose the constant $C$ in \eqref{frac-u-C} large enough, depending on $\|u\|_{L^\infty(B_1)}$, $\|f\|_{L^\infty(B_1)}$, $n$, $p$, 
$\eps_0$ and $\eps_1$,  in such a way that
 $\|\hat{f}\|_{L^\infty(B_1)}\le \min{\left\{\left(\frac{\eta\varepsilon_0}{2}\right)^p,\, \ep_1,\, 1\right\}}$ and $\|\hat{u}\|_{L^\infty(B_1)}\le 1$, where $\eta$ is the constant given in Lemma 
\ref{weak-harnack-p-lap} and $\eps_0$ and $\ep_1$ are small universal constants, depending on $\sigma$, to be made precise later. 

We notice that the function $\hat{G}$ giving the free boundary condition for $\hat{u}$,   also satisfies $\eqref{G2}$, for $t>\frac{M}{C}$, $x\in B_1$.

Then, choosing the constant $C$ properly (depending on $M$ and $m$ as well) and renaming the functions to simplify the notation, we may assume that 
\begin{equation}\label{bounds-u-f}
\|u\|_{L^\infty(B_1)} \leq 1, \qquad \|{f}\|_{L^\infty(B_1)}\le \min{\left\{\left(\frac{\eta\varepsilon_0}{2}\right)^p,\, \ep_1,\, 1\right\}}
\end{equation}
 and that $G$ satisfies 
\begin{equation}\label{G2m}
\sigma t\leq G(t, x)\leq \sigma^{-1}t,\quad \mbox{for}\:\:t>m,\: x\in B_1,
\end{equation}
 with $m >0$ a  small universal constant, depending on $\sigma$, to be made precise later. 

\smallskip

We will next prove the following claim:

\smallskip

{\it Claim.} There exists a constant $0<\delta<1$ universal, depending on $\sigma,$ such that
\begin{equation}\label{claim-delta}
 \text{if $\|u\|_{L^\infty(B_1)} \leq 1$ then $\|u\|_{L^\infty(B_\delta)} \leq 1-\delta.$}
\end{equation}
Let us observe that, once the claim \eqref{claim-delta} is established, we get by rescaling that, if $\|u\|_{L^\infty(B_1)} \leq 1$, then 
\begin{equation}\label{gives-holder} 
\|u\|_{L^\infty(B_r)} \leq r^\alpha, \quad \text{for $r=1, \delta, \delta^2, \ldots,$}
\end{equation}
for $0<\alpha<1$ universal, $\alpha=\alpha (\delta)$.
In fact, we define, for $x \in B_1$,
$$\tilde{u}(x):= \frac{u(r x)}{r^\alpha}, \quad  \tilde{f}(x):= r^{p(1-\alpha)+\alpha}{f(rx)}, \quad \tilde{G}(t,x) := \frac{G(r^{\alpha-1}t, rx)}{r^{\alpha-1}}.$$
Then, $\tilde{u}$ is a viscosity solution to
\begin{equation*}  
\left\{
\begin{array}{ll}
\Delta_{p} \tilde{u} = \tilde{f}, & \hbox{in $B_1^+(\tilde{u})\cup B_1^-(\tilde{u})$}, \\
\  &  \\
{\tilde{u}}_\nu^+=\tilde{G}({\hat{u}}_\nu^-,x), & \hbox{on $F(\tilde{u})$}, 
\end{array}
\right.
\end{equation*}
where the functions $\tilde{u}$ and $\tilde{f}$ also satisfy \eqref{bounds-u-f}, and the  function $\tilde{G}$ giving the free boundary condition for $\tilde u$,   also satisfies $\eqref{G2m}$. Then, \eqref{gives-holder}  follows from \eqref{claim-delta} by induction.

\smallskip

There holds  that \eqref{gives-holder} implies the H\"older continuity of $u$.

\smallskip

To prove the claim \eqref{claim-delta}, we keep in mind \eqref{bounds-u-f}, and we observe first that $0 \leq u^\pm \leq 1$ and that $u^+$ and $u^-$ are viscosity subsolutions to $\Delta_{p} u=-|f|$ in $B_1$ (see, for instance, Proposition 2.8 in \cite{CC}). We here recall that the notions of viscosity and weak subsolution of the $p$-Laplacian are equivalent (see Remark \ref{equiv-not}).  Therefore, by Lemma \ref{weak-harnack-p-lap}  either of the following situations holds, for any $0<r\le 1$, with $0<\eta<1$ a universal constant:
\be\label{weak1}\text{If \quad $\frac{|\{u^- = 0\} \cap B_{r/2}|}{|B_{r/2}|} \geq \frac 12 \quad $ then \quad $\sup_{B_{r/2}} u^- \leq (1-\eta)\sup_{B_r}u^- + ||f||^{1/p}_{L^{\infty}(B_r)}$}
\ee
\be\label{weak2} \text{If \quad $\frac{|\{u^+ = 0\} \cap B_{r/2}|}{|B_{r/2}|} \geq \frac 12 $\quad  then \quad $\sup_{B_{r/2}} u^+ \leq (1-\eta) \sup_{B_r}u^+ + ||f||^{1/p}_{L^{\infty}(B_r)}$.}
\ee
We will next apply this alternative to a sequence of radii, $r_k= 2^{-k}$, $k=0,1,\ldots$. In fact, assume that  \eqref{weak1} holds at $k=0$. Two different cases can occur.

Fix $\bar k$ universal, depending on $\sigma$, to be determined later.

\smallskip

{\it Case 1.} For some $k \leq\bar k$, \eqref{weak2} holds for $r=r_k.$ Then the claim \eqref{claim-delta} is an immediate consequence.

\smallskip

{\it Case 2.} Every $k \leq \bar k$, satisfies \eqref{weak1}. Hence,
\begin{equation}\label{weak3} u^- \leq (1-\eta)^{\bar k} + \frac{1}{\eta}||f||^{1/p}_{L^{\infty}(B_1)}\leq 
(1-\eta)^{\bar k} + \frac{\eps_0}{2}  \leq \eps_0, \quad u^+ \leq 1 \quad \text{in $B_{r_{\bar k}}$,}\ee
where $\eps_0$ is the small universal constant, depending on $\sigma$, appearing in \eqref{bounds-u-f}, that will  be determined later, and $\bar k$ is large enough depending on  the chosen $\eps_0$. 

Let us  verify that,  in this case, 
\begin{equation}\label{u+1/2}
 u^+ \leq \frac 1 2 \quad \text{in $B_{r_{\bar k}/4}$.}
\end{equation}
Once \eqref{u+1/2} is obtained,  the claim \eqref{claim-delta} will follow in this case as well.

We assume by contradiction that, for some $x_0 \in B_{r_{\bar k}/4}$, there holds $$u(x_0) > \frac 1 2.$$ 
We denote  $B_d(x_0)$  the largest ball around $x_0$ which is contained in $B^+_{1}(u)$. Namely, there exists $y_0 \in F(u)$ such that 
$$d={\rm dist} (x_0, F(u))=|x_0-y_0| \leq r_{\bar k}/4.$$

Next, applying Harnack inequality (see, for instance, Theorem 1.1 and Corollary 1.1 in \cite{T}), we get
\begin{equation*}
\sup_{B_{d/2}(x_0)} u \leq C(\inf_{B_{d/2}(x_0)}u + \theta),
\end{equation*}
where $C$ depends only on $n$ and $p$ and
\begin{equation*}
\theta=\mu \frac{d}{2} + (\mu \frac{d}{2})^{p/(p-1)},\qquad \mu=||f||^{1/p}_{L^{\infty}(B_d(x_0))}.
\end{equation*} 
But
\begin{equation*}
\theta \le d \le r_{\bar k}/4 \le \frac{1}{4C},
\end{equation*}         
if we choose $\bar{k}$ so that, in addition, $\bar{k}\ge k_0(n,p)$. Then,
\be\label{harnack}u \geq c_0 \quad \text{in $B_{d/2}(x_0)$},\ee
with $c_0=c_0(n,p)>0$. We now define
\begin{equation*}
 \psi(x):= \begin{cases}c(|x-x_0|^{-\gamma} - d^{-\gamma}) & \text{if $|x-x_0| \geq d/2$}\\
c_0 & \text{if $|x-x_0| \leq d/2$,}
\end{cases}
\end{equation*}
where
\begin{equation}\label{choice-of-c}
c=\frac{c_0 d^{\gamma}}{2^{\gamma}-1}.
\end{equation}
Here  $c$ is chosen in order that $\psi$ be continuous on $\p B_{d/2}(x_0)$. Next, we denote  $\lambda_0=\min\{1,p-1\}$, $\Lambda_0=\max\{1,p-1\}$, and we observe that, for $x\notin B_{d/2}(x_0)$,
\begin{equation*}
\mathcal M^-_{\lambda_0,\Lambda_0}(D^2 \psi)
 =  c \gamma|x- x_0|^{-\gamma-2}(\lambda_0(\gamma+1) -\Lambda_0(n-1)) \end{equation*}
(see, for instance, Lemma 2.5 in \cite{DFS2}). We now fix $\gamma$,
\begin{equation*}   
\gamma=\gamma(n,p)=\max{\left\{\frac{2\Lambda_0}{\lambda_0}(n-1)-1,1\right\}}.
\end{equation*}

Then, using  \eqref{p(x)-vs-pucci} and \eqref{choice-of-c} we obtain that, in $B_{2d}(x_0) \setminus \overline{B_{d/2}(x_0)}$,
\begin{equation*}
\begin{aligned}
\Delta_{p} \psi &\ge |\nabla \psi|^{p-2}\mathcal{M}_{\lambda_0,\Lambda_0}^-(D^2\psi)\\
&\ge (c\gamma |x- x_0|^{-\gamma-1})^{(p-2)}
c \gamma|x- x_0|^{-\gamma-2}(\lambda_0(\gamma+1) -\Lambda_0(n-1))\\
&\ge (c\gamma)^{p-1} |x- x_0|^{\gamma(1-p) -p} \Lambda_0(n-1) 
\ge (\frac{c_0 d^{\gamma}}{2^{\gamma}-1}\gamma)^{p-1}(2d)^{\gamma(1-p) -p} \Lambda_0(n-1)\\
&\ge(\frac{c_0 }{2^{\gamma}-1}\gamma)^{p-1}2^{\gamma(1-p) -p} d^{-p}\Lambda_0(n-1)
\ge (\frac{c_0 }{2^{\gamma}-1}\gamma)^{p-1}2^{\gamma(1-p) -p} \Lambda_0(n-1).
\end{aligned}
\end{equation*}

We now define
$$w:= \psi^+ - \frac \sigma 2 \psi^-,$$
with $\sigma$ the constant in \eqref{G2m},   $$D:= B_{2d}(x_0) \setminus \overline{B_{d/2}(x_0)} \subset B_{r_{\bar k}},$$
and we choose
\begin{equation*}
\varepsilon_1=\varepsilon_1(n,p,\sigma)= \min{\left\{1,\left(\frac{\sigma}{2}\right)^{(p-1)}\right\}} \, \left((\frac{c_0 }{2^{\gamma}-1}\gamma)^{p-1}2^{\gamma(1-p) -p} \Lambda_0(n-1)\right).
\end{equation*}

Then, recalling \eqref{bounds-u-f}, we obtain
\begin{equation*}
\Delta_{p} w \ge ||f||_{L^{\infty}(B_1)}\qquad \text{ in } D\cap\left(\{w>0\}\cup\{w<0\}\right).
\end{equation*}

Let us now show that
\begin{equation}\label{claim-on-D}
u \geq w \quad \text{ on } \ D,
\end{equation}
if $\eps_0$ is sufficiently small.

In fact, on one hand,  $u \geq w$ in $\overline{B_{d}(x_0)}$ follows as a consequence of  \eqref{harnack} and  the maximum principle.

On the other hand, it is not hard to check that $u \geq w$ in $\{u\geq 0\} \cap (B_{2d}(x_0) \setminus \overline{B_{d}(x_0)}).$ We next want to show that $u \geq w$ in the set $\{u< 0\} \cap (B_{2d}(x_0) \setminus \overline{B_{d}(x_0)}).$ To be able to employ  the maximum principle we only need to see that $u \geq w$ on $\p B_{2d}(x_0) \cap \{u<0\}.$ We use that, by \eqref{weak3} and \eqref{choice-of-c}, in this set we have $u \geq -\eps_0$ and 
$$
w=-\frac{\sigma}{2}\, \frac{c}{d^{\gamma}}\, \frac{2^{\gamma}-1}{2^{\gamma}}=-\frac{\sigma}{2}\,  \frac{c_0}{2^{\gamma}}.
$$
 Therefore, it is enough to fix $\eps_0$ small universal, depending on $\sigma$, and then  \eqref{claim-on-D} follows.

We now observe that, since \eqref{claim-on-D} holds and $u$ and $w$ touch at $y_0 \in F(u)$, then, by Definition \ref{defnhsol1bis}-(ii)(1), we have 
$$|\nabla \psi(y_0)|\le G(\frac \sigma 2|\nabla \psi (y_0)|, y_0).$$

However,  the application of \eqref{G2m} and \eqref{choice-of-c} yields 
$$|\nabla \psi(y_0)|> G(\frac \sigma 2|\nabla \psi (y_0)|, y_0),$$
provided $m$ is chosen small enough universal, depending on $\sigma$. Namely,
$$\frac \sigma 2 |\nabla \psi (y_0)|= \gamma \frac \sigma 2 c d^{-\gamma-1}> \gamma \frac{\sigma}{2}\, \frac{c_0}{2^{\gamma}-1}> m.$$
This gives a contradiction and proves \eqref{u+1/2}, which implies the claim \eqref{claim-delta}.

\smallskip

In order to conclude the proof, let us now suppose that $0 \not \in F(u).$ If $B_{1/2} \cap F(u) =\emptyset$,  we can apply  interior estimates for the $p$-Laplace operator (see, for instance, Proposition 2.2 in \cite{Fan}). If there is ${\hat{x}} \in B_{1/2} \cap F(u)$, the argument above in $B_{1/2}(\hat{x})$ gives the
H\"older bound in $B_{1/4}(\hat{x}).$ We finally use a covering argument and in this way we obtain the stated result.
\end{proof}

\section{Lipschitz continuity}\label{section5}

In this section we prove the local Lipschitz continuity of viscosity solutions to \eqref{fbtrue}, namely, Theorem \ref{Lipschitz_cont}. 
This proof will make use of regularity theorems obtained in Section \ref{section3} (i.e., Theorems \ref{main_new-nondeg-general-rhs0}, \ref{main_new-nondeg-general-p=2} and \ref{main_new-nondeg-general-G(t)}).
We emphasize that the application of these results requires a careful preliminary preparation (see Remark \ref{further_remark}).

\smallskip

We will first prove an approximation lemma that will be a key device for Theorem \ref{Lipschitz_cont}.

\begin{lem}\label{limit_problem_second}
Let $u_k$ be a  viscosity solution to 
\begin{equation*}  
\left\{
\begin{array}{ll}
\Delta_{p} u_k = f_k, & \hbox{in $B_1^+(u_k)\cup B_1^-(u_k)$}, \\
\  &  \\
(u_k)_\nu^+=G_k((u_k)_\nu^-,x), & \hbox{on $F(u_k)$}, 
\end{array}
\right.
\end{equation*}
with  $f_k\in L^{\infty}(B_1)$ and continuous in $B_1^+(u_k)\cup B_1^-(u_k)$  and assume that  $G_k$ satisfies (P1) in $B_1$. Assume moreover that 
\begin{equation}\label{convergenceslabel}
u_k\to u^*,
\end{equation}
\begin{equation*}
f_k\to 0,
\end{equation*}
and
\begin{equation}\label{convergenceslabel3}
G_k(x,t)\to G^*(t),
\end{equation}
with
\begin{equation}\label{convergenceslabel4}
 G^*(t)=t,
\end{equation}
where the convergences hold uniformly on compact sets. Then,  
$$
\Delta_pu^*=0\quad \mbox{in}\:\:B_1.
$$
\end{lem}

\begin{proof}  Arguing in a similar way as in \cite{FL2}, Theorem 1.2, Step III (letting $p_k(x)\equiv p$ in that proof), we show  that 
\begin{equation*}
\Delta_pu^*=0\ \quad \text{in $B_1 \cap \{u^*\neq 0\}$}.
\end{equation*}

We will next prove that the equation is satisfied across $\{u^*=0\}.$ In fact, we have to see that if $P$ is a quadratic polynomial with $\Delta_pP >0,$ then $P$ can not touch $u^*$ strictly from below at a point $x^* $ where $u^*(x^*)=0$  and $\nabla P(x^*) \neq 0.$  Suppose by contradiction that there exists such a point.

 Without loss of generality, we will assume that $\nabla P(x^*)= \gamma e_n$, with $\gamma >0.$

We define
$$\ \psi:=(1+\eps) P^+ - P^-.$$
For some $\eps>0$ small, $\psi$ still separates strictly from $u^*$ on  $\p B_\rho (x^*)$, for  $\rho>0$ small enough, and coincides with it at $x^*$.
Let us consider $$\psi_t(x)= \psi(x+ t e_n), \quad x \in B_\rho(x^*).$$ Then, for $t=-C \eps$, $C>0$ large, we have that $\psi_t$ is strictly below all $u_k$'s with $k$ large enough, since \eqref{convergenceslabel} holds.  We increase $t$ until a small $c_0 >0$ to guarantee that $\psi_t$ crosses $u^*$ and  all the $u_k$'s, with $k$ large. Hence, $\psi_t$ must touch the $u_k$'s for the first time at $t=t_k$ small. Since the separation of $\psi$ and $u^*$ on $\p B_\rho(x^*)$ is strict, the first touching point $x_k$ can not occur there (if $c_0$ is small depending on the separation on $\p B_\rho(x^*)$).
Recalling that $\Delta_pP>0$  and $f_k\to 0$ uniformly, as $k\to \infty,$ we conclude that $x_k \in F(u_k).$ However, taking into account \eqref{convergenceslabel3} and \eqref{convergenceslabel4}, we deduce that $$(1+\eps)|\nabla P(x_k+ t_k e_n)| > G_k(|\nabla P(x_k+ t_k e_n)|, x_k+ t_k e_n).$$ Then, we get a contradiction to Definition \ref{defnhsol1bis} (ii)-(1) for $u_k$ and the result follows.
\end{proof}

The main tool in the proof of Theorem \ref{Lipschitz_cont} is the following result

\begin{lem}\label{boundness}
Let $u$ be a viscosity solution to \eqref{fbtrue} in $B_2$. Assume $G$ satisfies assumptions (P2) and (P3) in $B_1$ and $0\in F(u)$. If $f\not\equiv 0$ and $p\neq 2$ we assume that (P4) also holds.

There exist positive constants $L_0$, $\delta$ and $C$ (depending only on $n$, $p$, $G$ and $\|f\|_{L^\infty(B_2)}$) such that one of the following alternative holds:
\begin{itemize}
\item[(i)] $u$ is Lipschitz in $B_{\delta}$ and $|\nabla u|\leq C\max\{\|u\|_{L^{\infty}(B_1)}, L_0\}$ in $B_{\delta}$.
\item[(ii)] $\frac{1}{\delta}\|u\|_{L^{\infty}(B_\delta)}\leq \frac{1}{2}\max\{\|u\|_{L^{\infty}(B_1)}, L_0\}$.
\end{itemize}
\end{lem}

\begin{proof}  Let $\delta>0$ be fixed,  to be precised later. We will assume by contradiction that there exist a sequence of constants $L_k \to \infty$, and a sequence of  solutions $u_k$ to 
\begin{equation*} 
\left\{
\begin{array}{ll}
\Delta_{p} u_k = f_k, & \hbox{in $B_2^+(u_k)\cup B_2^-(u_k)$}, \\
\  &  \\
(u_k)_\nu^+=G((u_k)_\nu^-,x), & \hbox{on $F(u_k)$}, 
\end{array}
\right.
\end{equation*}
  in $B_2$, with $0\in F(u_k)$ and $\|f_k\|_{L^\infty(B_2)}\le M_0$ for some constant $M_0$, such that
 $u_k$ does not satisfy neither option $(i)$ nor $(ii).$
 Let $$C_k:=\max\{\|u_k\|_{L^\infty(B_1)}, L_k\} $$ and 
$$\tilde u_k := \frac{u_k}{C_k},\quad \tilde{f}_k:=\frac{f_k}{C_k^{p-1}}, \quad \tilde G_k(t,x):= \frac{G(C_k t,x)}{C_k}.$$
Then, every $\tilde{u}_k$ is a viscosity solution to
\begin{equation*}  
\left\{
\begin{array}{ll}
\Delta_{p} \tilde{u}_k = \tilde{f}_k, & \hbox{in $B_2^+(\tilde{u}_k)\cup B_2^-(\tilde{u}_k)$}, \\
\  &  \\
(\tilde{u}_k)_\nu^+=\tilde{G}_k((\tilde{u}_k)_\nu^-,x), & \hbox{on $F(\tilde{u}_k)$},
\end{array}
\right.
\end{equation*}
with
\begin{equation*}  
  \|\tilde{u}_k\|_{L^\infty(B_1)} \le 1
\end{equation*}
and
\begin{equation}\label{tildefk}
\|\tilde{f}_k\|_{L^\infty(B_1)} \to 0, \quad \text{ as }k\to \infty.
 \end{equation}

 Now, using  (P2) (first assumption) and (P3), we can apply Theorem \ref{holder_reg} and conclude  that, as $k\to \infty$, for a subsequence,
\begin{equation*}  
   \tilde G_k(t,x) \to  G^*(t), \quad \text{$G^*(t)=t$},
\end{equation*}
 uniformly on compact sets of $(0,\infty)\times B_2$ and
\begin{equation*}
	\tilde u_k \to u^*,
 \end{equation*} 
uniformly on compact sets of $B_1$.

We can thus employ the compactness  Lemma \ref{limit_problem_second} to deduce that  
\begin{equation*}
\Delta_p u^* = 0 \quad \text{in $B_1$}, \quad \text{ with } \ \|u^*\|_{L^\infty(B_1)} \le 1.
\end{equation*}

Next, from the local $C^{1,\alpha}$ estimates for the $p$-Laplace operator  (see, for instance, \cite{Fan}, Theorem 1.1), we conclude that 
\be\label{c1}\|u^* - l\|_{L^\infty(B_r)} \leq C r^{1+\alpha}, \quad \text{ for all }  r \leq 1/2.\ee
Here  $C$ and $\alpha$ depend only on $n$ and $p$, and $l(x)=a \cdot x$ for a vector $a \in \R^n$, with $|a| \leq C$.

Two different cases can occur: 

\smallskip

{\it Case 1.} $|a| \leq \frac 1 4.$

In this case, \eqref{c1} yields $$\frac {1}{\delta}|u^*| \leq \frac 1 4 + C\delta^\alpha \leq \frac 1 3 \quad \text{in $B_\delta$,}$$
if $\delta\le 1/2$ is chosen small enough depending only on $n$ and $p$. 
In this way, all $u_k$'s with $k$ large satisfy $(ii)$, which gives a contradiction.

\smallskip

{\it Case 2.} $|a| > \frac 1 4$

In this case we will apply  flatness regularity results, derived from \cite{FL3}, that we developed in Section \ref{section3} (see Theorems \ref{main_new-nondeg-general-rhs0}, \ref{main_new-nondeg-general-p=2} and
\ref{main_new-nondeg-general-G(t)}, and Remark \ref{further_remark}).

In fact, using that $\tilde u_k$ converges uniformly to $u^*$ and \eqref{c1} holds, we obtain, for  every $k$ large,
\be\label{c1k}|\tilde u_k - a\cdot x| \leq \hat{C}\delta^{1+\alpha}, \quad \text{in $B_{2\delta}$},\ee
with $\hat{C}=\hat{C}(n,p)$. Let us now define
$$\beta_k:=|a|, \quad \alpha_k = \tilde G_k(|a|,0), \quad \omega:= \frac{a}{|a|}.$$

Then, from \eqref{c1k}, recalling that $\tilde G_k$ converges  to the identity uniformly on compact sets, we deduce
$$|\tilde u_k - U_{\beta_k}| \leq 2\hat{C}\delta^{1+\alpha}, \quad \text{in $B_{2\delta}$},$$
for $k$ large, where
$$U_{\beta_k}(x):= \alpha_k(x \cdot \omega)^+ - \beta_k(x \cdot \omega)^-.$$

Now, employing \eqref{tildefk}, we conclude from Theorems \ref{main_new-nondeg-general-rhs0}, \ref{main_new-nondeg-general-p=2} and \ref{main_new-nondeg-general-G(t)} (see Remark \ref{further_remark}) that, for large $k$,
$$ \tilde u_k\in C^{1,\gamma}(\overline{B^+_{\delta}(\tilde u_k)})\cap C^{1,\gamma}(\overline{B^-_{\delta}(\tilde u_k)}),$$ 
with $C^{1,\gamma}$ norms bounded by a universal constant, if $\delta$ is chosen small enough universal.

As a consequence, the $\tilde u_k$'s are uniformly Lipschitz with universal constant. Therefore,
$$|\nabla u_k| \leq C C_k \quad \text{in $B_\delta$},$$
for $k$ large.
This is in contradiction with the fact that the $u_k$'s do not satisfy $(i)$ and concludes the proof.
\end{proof}

We are finally ready to prove our main result

\begin{proof}[\bf Proof of Theorem \ref{Lipschitz_cont}] 
Let $L_0$, $\delta$ and $C$ be  the universal constants in Lemma \ref{boundness}.

We first assume $0 \in F(u)$ and define
$$L:= \max\{\|u\|_{L^\infty(B_{3/4})},L_0\}$$
and
$$a(r):=\frac 1 r \|u\|_{L^\infty(B_r)}, \quad r \leq 3/4.$$
We want to obtain  that
\be\label{claim1} a(\delta^k) \leq C L, \quad \forall k\geq 1.\ee

In fact, by Lemma \ref{boundness} either alternative $(i)$ or $(ii)$ holds. 

In case alternative $(i)$  holds, $u$ is Lipschitz in $B_\delta$ and 
$$|\nabla u| \leq C L \quad \text{in $B_{\delta}$}.$$
 This shows that \eqref{claim1} is  satisfied for every $k \geq1.$

In case alternative $(ii)$ holds, we get 
$$ a(\delta) \leq \frac 1 2 \max\{\|u\|_{L^\infty(B_{3/4})},L_0\}  \leq L.$$
Let us next rescale and iterate. We define, for $k\geq 1$ and $x\in B_1$,
$$u_k(x):= \frac{u(\delta^k x)}{\delta^k},\quad {f}_k(x):=\delta^k f(\delta^k x), \quad  G_k(t,x):= G(t,\delta^kx).$$

Then, every ${u}_k$ is a viscosity solution to
\begin{equation*} 
\left\{
\begin{array}{ll}
\Delta_{p} {u}_k = {f}_k, & \hbox{in $B_1^+({u}_k)\cup B_1^-({u}_k)$}, \\
\  &  \\
({u}_k)_\nu^+={G}_k(({u}_k)_\nu^-,x), & \hbox{on $F({u}_k)$}. 
\end{array}
\right.
\end{equation*}
Moreover,
$$\|{f}_k\|_{L^\infty(B_1)}\le \|{f}\|_{L^\infty(B_1)}.$$

Then,  Lemma \ref{boundness} applies to every $u_k$ and thus, each  $u_k$ satisfies the conclusion of this lemma.

Suppose first that the $u_k$'s satisfy indefinitely the  alternative $(ii)$ of Lemma \ref{boundness}. Therefore,  
\be\label{claim2}a(\delta^k) \leq L, \quad \forall k\geq 1,\ee
and then, \eqref{claim1} holds.

Otherwise, let us denote  $\bar k \geq 1$  the smallest $k$ for which $u_k$ does not satisfy $(ii)$. It is not hard to see that \eqref{claim2} holds for all $1 \leq k \leq \bar k$ and the same happens with \eqref{claim1}.  Moreover,  
$u_{\bar k}$ satisfies the  alternative $(i)$ of Lemma \ref{boundness}, which implies that $u_{\bar k}$ is Lipschitz in $B_{\delta},$ with
$$|\nabla u_{\bar k}| \leq C \max\{\|u_{\bar k}\|_{L^\infty(B_{3/4})}, L_0\} \quad \text{in $B_{\delta}.$}$$
Now, \eqref{claim2} for $k=\bar k$ gives
$$|\nabla u_{\bar k}| \leq C \max\{\frac{1}{\delta}\|u_{\bar k-1}\|_{L^\infty(B_\delta)}, L_0\} \leq CL \quad \text{in $B_\delta$.}$$
Therefore \eqref{claim1} holds for every $k \geq \bar k +1$ as well.

  We now take $0<r<3/4$ and $k$ such that ${\delta}^{k+1}\le r \le {\delta}^{k}$. Then \eqref{claim1} implies
$$||u||_{L^{\infty}(B_r)} \leq ||u||_{L^{\infty}(B_{{\delta}^{k}})}\le {CL}{\delta}^{k}\le \frac{CL}{\delta} r.$$
The Lipschitz continuity of $u$ in $B_{1/2}$ then follows, which is the desired result.
\end{proof}

\section*{Acknowledgment }
The authors wish to thank Ovidiu Savin for pointing them out the subject of this paper and also for some helpful discussions.

\end{document}